\documentclass[a4paper,11pt]{article}

\usepackage{amsfonts}
\usepackage{amssymb}
\usepackage{amsthm}
\usepackage{amsmath}
\usepackage{a4wide}
\usepackage{mathrsfs}
\usepackage{epsfig}
\usepackage{palatino}
\usepackage{tikz,fp,ifthen}
\usepackage{float}

\newcommand{\pdfgraphics}{\ifpdf\DeclareGraphicsExtensions{.pdf,.jpg}\else\fi}
\usepackage{graphicx}

\usepackage{color}
\definecolor{citegreen}{rgb}{0,0.6,0}
\definecolor{refred}{rgb}{0.8,0,0}
\usepackage[colorlinks, citecolor=citegreen,linkcolor=refred]{hyperref}

\usepackage{mathtools}
\mathtoolsset{showonlyrefs}

\numberwithin{equation}{section}

\theoremstyle{plain}

\newtheorem{teo}{Theorem}[section]
\newtheorem{lemma}[teo]{Lemma}
\newtheorem{prop}[teo]{Proposition}
\newtheorem{cor}[teo]{Corollary}

\theoremstyle{definition}
\newtheorem{dfnz}[teo]{Definition}

\theoremstyle{remark}
\newtheorem{rem}[teo]{Remark}

\numberwithin{equation}{section}

\def\eps{\varepsilon}

\def\limup{\operatornamewithlimits{\overline{\lim}}}
\def\loc{_{\operatorname{loc}}}
\def\NN{\mathbb N}

\def\R{\mathbb R}

\renewcommand{\t }{\tau }

\newcommand{\intbar}{\etaathop{\int\etaakebox(-13.5,0){\rule[4pt]{.7em}{0.3pt}}
\kern-6pt}\nolimits}

\newcommand{\be}{\begin{equation}}
\newcommand{\ee}{\end{equation}}
\newcommand{\bea}{\begin{equation*}}
\newcommand{\eea}{\end{equation*}}

\def\loc{_{\operatorname{\rm loc}}}

\def\R{{{\mathbb R}}}
\def\SS{{{\mathbb S}}}

\def\NN{{{\mathbb N}}}

\def\eps{\varepsilon}

\def\tt{\mathfrak t}
\def\ss{\mathfrak s}
\def\comp{\circ}

\def\be{\begin{equation}}
\def\ee{\end{equation}}
\def\bea{\begin{eqnarray*}}
\def\bean{\begin{eqnarray}}
\def\eean{\end{eqnarray}}
\def\eea{\end{eqnarray*}}

\begin{document}
\pdfgraphics 

\title{Motion by curvature of networks with two triple junctions}

\author{Carlo Mantegazza \footnote{Dipartimento di Matematica e Applicazioni, Universit\`a di Napoli Federico II, Via Cintia, Monte S. Angelo
80126 Napoli, Italy} \and Matteo Novaga \footnote{Dipartimento di Matematica, Universit\`a di Pisa, Largo
    Bruno Pontecorvo 5, 56127 Pisa, Italy} \and Alessandra Pluda\footnotemark[2]}
\date{}

\maketitle

\begin{abstract}
\noindent We consider the evolution by curvature of 
a general embedded network with two triple junctions. 
We classify the possible singularities  and we discuss the long time existence
of the evolution.
\end{abstract}


\tableofcontents

\section{Introduction}

In this paper we study the motion by curvature of connected networks with two triple junctions
in an open and (strictly) convex set $\Omega\subset\mathbb{R}^2$, with fixed end--points on the boundary $\partial\Omega$.
During the evolution we require that the curves remain concurring at the triple junctions
forming angles of $120$ degrees (Herring condition).
As the evolution can be regarded as the gradient flow of the length functional,
the Herring condition turns out to be the natural one from the variational point of view and is related to the local stability of the triple junctions. We will call {\em regular} the networks with only triple junctions each one satisfying the Herring condition.

The evolution by curvature of a planar simple closed curve is by now completely understood: the curve becomes eventually convex 
and then shrinks to a point in finite time, asymptotically approaching the shape of a round circle (see~\cite{gaha1,gray1}).

Concerning the evolution of a general planar network the problem is far to be fully solved. A first short time existence result 
was proved by Bronsard and Reitich in~\cite{bronsard} for an initial regular network with only one triple junction and 
three end--points on the boundary of the domain (a so--called ``triod'') with Neumann boundary conditions, and then adapted for the case with Dirichlet boundary conditions in~\cite{mannovtor}.
An analogous theorem for a general network is shown  in~\cite{mannovplusch} and says that
for any initial smooth regular network there exists a smooth flow by curvature in a positive maximal time interval $[0,T)$. 

About the global behavior of the flow, in~\cite{MMN13,mannovtor} the authors study the evolution by curvature of a triod, showing that 
if the lengths of the three curves are bounded away from zero during the evolution,
then the evolution is smooth for every time and the triod tends to the unique Steiner configuration connecting the three fixed end--points. 

The simplest case of a network with a loop (a region bounded by one or more curves) is treated in~\cite{pluda}: a network 
composed by two curves, one of them closed, meeting only at one point. In this case, if the length of the non--closed curve is bounded away 
from zero during the evolution, the closed curve shrinks to a point after a finite time which depends only on the area initially enclosed in the loop.

In this paper, we consider networks with exactly two triple junctions and we obtain an almost complete description of the evolution till the 
appearance of the first singularity, adapting the techniques of~\cite{MMN13,pluda}.

The major open problem in the general context of the motion by curvature of networks, is the so--called multiplicity--one conjecture:
if the initial network $\mathbb{S}_0$ is embedded, not only $\mathbb{S}_t$ remains embedded for 
all the times, but also every possible $C^1\loc$--limit of rescalings of networks of the flow is an
embedded network with multiplicity one. This is a crucial ingredient in classifying blow--up limits of the flow, 
which is the main method to understand the singularity formation.\\
In Section~\ref{dsuL} we will introduce a geometric quantity which, by means of a monotonicity property it satisfies, can be used to prove the multiplicity--one conjecture in the case of networks with at most two triple junctions (similar quantities have already been used by Hamilton in~\cite{hamilton3} and Huisken in~\cite{huisk3}). Unfortunately, this argument cannot be extended to networks
with more that two triple junctions, as the analogous quantity does not share such monotonicity property anymore.

The following theorem and Proposition~\ref{kscoppia} are the main result of the paper, describing the behavior of the network at the singular time.

\begin{teo}\label{main}
Let $\Omega\subset\mathbb{R}^2$ be a smooth, strictly convex, open set. 
Let $\mathbb{S}_0$ be a compact initial network with two triple junctions and  with possibly
fixed end--points on $\partial\Omega$, and
let $\mathbb{S}_t$ be the smooth evolution by curvature
of $\mathbb{S}_0$ in a maximal time interval $[0,T)$.

Then,  if the network $\mathbb{S}_0$ has at least one loop, then the maximal time of existence $T$ is finite and one of the following situations occurs:
\begin{enumerate}
\item the limit of the length of a curve that connects the two $3$--points goes to zero as $t\to T$, 
and the curvature remains bounded;
\item the limit of the length of a curve that connects the $3$--point with an end--point goes to zero as $t\to T$, 
and the curvature remains bounded;
\item  the lengths of the curves composing the loop go to zero as $t\to T$, and $\lim_{t\to T}\int_{\mathbb{S}_t}k^2\,ds=+\infty$.
\end{enumerate}

If the network is a tree and $T$ is finite, the curvature is uniformly bounded and only the first two situations listed above can happen. If instead $T=+\infty$, for every sequence of times $t_i\to+\infty$, there exists a subsequence (not relabeled) such that the evolving networks $\SS_{t_i}$ converge in 
$C^{1,\alpha}\cap W^{2,2}$, for every $\alpha\in(0,1/2)$,
to a possibly degenerate (see Section~\ref{degsec}) regular network with zero curvature (hence, ``stationary'' for the length functional), as $i\to\infty$.
\end{teo}

To prove this theorem, following~\cite{MMN13,mannovtor}, we analyze the blow--up limits arising from a sequence of rescaled networks (see Section~\ref{rescaling} and Proposition~\ref{resclimit}). Such limit regular networks satisfy 
the shrinkers equation, that is,
$$
\underline{k} + x^\perp=0,
$$ 
where $\underline{k} $ is the curvature vector and $x^\perp$ the normal component of the position vector $x$ and, assuming they have at most two triple junctions, a classification is complete (see~\cite{balhausman2}), thanks to the contributions in~\cite{balhausman,chenguo,haettenschweiler,schnurerlens} (we underline that for more complicated topological structure -- more than two triple junctions -- such a classification is not available at the moment). Then, by means of White's local regularity theorem in~\cite{white1}, the works~\cite{Ilnevsch,MMN13} and Proposition~\ref{cross}, we will see that if a flow has a flat blow--up limit around a point, its curvature is locally bounded.

As a consequence of this analysis and the classification of shrinkers, we have the following proposition dealing with the situation when the curvature is unbounded.

\begin{prop}\label{kscoppia}
Let $\mathbb{S}_t$ be as in Theorem \ref{main}.
If $\lim_{t\to T}\int_{\mathbb{S}_t}k^2\,ds=+\infty$, then 
there exists a point $x_0\in\Omega$ such that for a sequence of rescaled times
$\tt_j$, the associated rescaled networks $\widetilde{\mathbb{S}}_{x_0,\tt_j}$
defined in Section \ref{rescaling} tend
to one of the non--flat shrinkers with one or two triple junctions, that is, the sequence
$\widetilde{\mathbb{S}}_{x_0,\tt_j}$ converges in $C^{1,\alpha}\loc\cap W^{2,2}\loc$,
for every $\alpha\in(0,1/2)$, as $j\to\infty$, to:
\begin{enumerate}
\item a Brakke spoon;
\item a standard lens; 
\item a fish.
\end{enumerate}
\begin{figure}[H]
\begin{center}
\begin{tikzpicture}[scale=0.6]
\draw[color=black, scale=0.4, shift={(-30.45,0)}]
(-3.035,0)to[out= 60,in=180, looseness=1] (2.2,3)
(2.2,3)to[out= 0,in=90, looseness=1] (5.43,0)
(-3.035,0)to[out= -60,in=180, looseness=1] (2.2,-3)
(2.2,-3)to[out= 0,in=-90, looseness=1] (5.43,0);
\draw[color=black, scale=0.4, shift={(-30.45,0)}]
(-7,0)to[out= 0,in=180, looseness=1](-3,0);
\draw[color=black, scale=0.4, shift={(-30.45,0)}, dashed]
(-9,0)to[out= 0,in=180, looseness=1](-7,0);
\draw[color=black,scale=0.4, shift={(-10.5,0)}]
(-3.035,0)to[out= 60,in=180, looseness=1] (2.2,2.8)
(2.2,2.8)to[out= 0,in=120, looseness=1] (7.435,0)
(2.2,-2.7)to[out= 0,in=-120, looseness=1] (7.435,0)
(-3.035,0)to[out= -60,in=180, looseness=1] (2.2,-2.7);
\draw[color=black,scale=0.4, shift={(-10.5,0)}]
(-7,0)to[out= 0,in=180, looseness=1](-3.035,0)
(7.435,0)to[out= 0,in=180, looseness=1](11.4,0);
\draw[color=black,dashed,scale=0.4, shift={(-10.5,0)}]
(-9,0)to[out= 0,in=180, looseness=1](-7,0)
(11.4,0)to[out= 0,in=180, looseness=1](13.4,0);
\draw[color=black,shift={(4.733,0)},scale=1.5]
(-0.47,0)to[out= 20,in=180, looseness=1](1.5,0.65)
(1.5,0.65)to[out= 0,in=90, looseness=1] (2.37,0)
(-0.47,0)to[out= -20,in=180, looseness=1](1.5,-0.65)
(1.5,-0.65)to[out= 0,in=-90, looseness=1] (2.37,0);
\draw[white, very thick,shift={(4.733,0)},scale=1.5]
(-0.47,0)--(-.150,-0.13)
(-0.47,0)--(-.150,0.13);
\draw[color=black,shift={(4.733,0)},scale=1.5]
(-.150,0.13)to[out= -101,in=90, looseness=1](-.18,0)
(-.18,0)to[out= -90,in=101, looseness=1](-.150,-0.13);
\draw[black,shift={(4.733,0)},scale=1.5]
(-.150,0.13)--(-1.13,0.98)
(-.150,-0.13)--(-1.13,-0.98);
\draw[black, dashed,shift={(4.733,0)},scale=1.5]
(-1.13,0.98)--(-1.50,1.31)
(-1.13,-0.98)--(-1.50,-1.31);
\fill(-3.3,0) circle (2pt);
\fill(-12.33,0) circle (2pt);
\fill(4.733,0) circle (2pt);
\path[font=\small]
(-11.43,0.2) node[left]{$O$}
(5.7,0.1) node[left]{$O$}
(-2.43,0.2) node[left]{$O$};
\path[font=\small]
(-15.33,-3) node[right]{$1)\;\text{Brakke spoon}$}
(5,-3) node[right]{$3)\;\text{fish}$}
(-6,-3) node[right]{$2)\;\text{standard lens}$};
\end{tikzpicture}
\end{center}
\begin{caption}
{Non--flat regular shrinkers with one or two triple junctions.\label{clas0}}
\end{caption}
\end{figure}
\end{prop}

It is worth mentioning that by this proposition it follows that the network cannot completely vanish shrinking at a point, as $t\to T$ (Proposition~\ref{vanishing}).

We resume here the structure of the paper: in Section~\ref{classi} we set up the basic notation and definitions and we classify all topological types of networks with two triple junctions.
In Section~\ref{preliminary} we state the short time existence theorem and we introduce all the main properties and useful notions for the subsequential analysis.  Section~\ref{dsuL} is devoted to prove the multiplicity--one conjecture for evolving networks with at most two triple junctions.
Then, in Section~\ref{longtime}, after the analysis of the possible blow--up networks in different situations we prove Theorem~\ref{main}. We conclude the paper with a description of the way to possibly restart the flow after a singularity.

\section*{Acknowledgments}
The second and the third authors are partially supported by the University of Pisa grant PRA--2015--0017.

\section{Networks with two triple junctions and their curvature flow}\label{classi}

Given a $C^1$ curve $\sigma:[0,1]\to\R^2$ we say that it is {\em regular} if
$\sigma_x=\frac{d\sigma}{dx}$ is never zero. It is then well defined
its {\em unit tangent vector} $\tau=\sigma_{x}/|\sigma_{x}|$. We define its {\em unit normal vector} as
$\nu=R\tau=R\sigma_{x}/|\sigma_{x}|$, 
where $R:\mathbb{R}^{2}\rightarrow\mathbb{R}^{2}$ is the counterclockwise
rotation centered in the origin of $\mathbb{R}^{2}$ of angle
${\pi}/{2}$.\\
If the curve $\sigma$ is $C^2$ and regular, its  {\em curvature  vector} is well
defined as $\underline{k}=\tau_x/|\sigma_{x}|$.\\
The arclength parameter of a curve $\sigma$ is given by
$$
s=s(x)=\int_0^x\vert\sigma_x(\xi)\vert\,d\xi\,.
$$
Notice that $\partial_s=\vert\sigma_x\vert^{-1}\partial_x$, 
then $\tau=\partial_s\sigma$ and $\underline{k}=\partial_s\tau$, 
hence, the curvature of $\sigma$ is given by $k=\langle\underline{k}\,\vert\,\nu\rangle$, 
as $\underline{k}=k\nu$.

\begin{dfnz}\label{network2}
Let $\Omega$ be a smooth, convex, open set in $\mathbb{R}^{2}$. A
{\em network with two triple junctions} $\mathbb{S}=\bigcup_{i=1}^{n}\sigma^{i}([0,1])$ in
$\Omega$ is a connected set in the plane described by a finite family
of $C^1$, regular curves $\sigma^{i}:[0,1]\to\overline{\Omega}$ 
such that
\begin{enumerate}
\item the relative interior of every curve $\sigma^{i}$, that is 
$\sigma^i(0,1)$, is embedded (hence, it has no self--intersections); 
a curve can self--intersect itself only possibly ``closing'' at its end--points;
\item two different curves can intersect each other only at their end--points;
\item there are exactly two points $O^{1},O^{2}\in\Omega$,
  the {\em 3--points} of the network, where three different curves intersect each other or where
  a curve intersects a different curve that has a self--intersection;
\item any curve can ``touch'' the boundary of $\Omega$ only at one of its
  end--points and if a curve of the network touches the boundary of $\Omega$ at a
  point $P$, no other curve can end in the same point $P$.
  We call {\em end--points} of the network, the vertices
  $P^{l}$ of $\mathbb{S}$ on the boundary of $\Omega$.
\end{enumerate}

We call the network {\em regular} if at the two 3--points $O^{1}$ and $O^{2}$ the sum of the exterior unit tangent vectors of the concurring curves is equal to zero ({\em Herring condition}).
  
We say that a network is of class $C^{k}$ or $C^{\infty}$ if all its curves 
are respectively of class $C^{k}$ or $C^{\infty}$.
\end{dfnz}

First we focus on the topological classification of these networks.

As just seen in Definition~\ref{network2}, 
we parametrize the curves composing the network with $\sigma^i:[0,1]\to\mathbb{R}^2$.
In each $3$--point either concur three different non--closed curves 
(for instance $O^1=\sigma^1(0)=\sigma^2(0)=\sigma^3(0)$)
or two curves, one of which closed (that is $O^1=\sigma^1(0)=\sigma^1(1)=\sigma^2(0)$).
If a curve is not closed (hence $\sigma^1(0)\neq \sigma^1(1)$), 
there are only two possibilities for its end--point not concurring in $O^1$:
either to be an end--point of the network on the boundary of $\Omega$, or to be in the other triple junction $O^2$.
If we repeat the above reasoning for every end--point, we obtain 
all cases shown in the following figure.
\begin{figure}[H]
\begin{center}
\begin{tikzpicture}[scale=1.48]
\draw[scale=0.45,shift={(8,-4.5)}] 
(-1.73,-1.8) 
to[out= 180,in=180, looseness=1] (-2.8,0) 
to[out= 60,in=150, looseness=1.5] (-1.5,1) 
(-2.8,0)
to[out=-60,in=180, looseness=0.9] (-1.25,-0.75)
(-1.5,1)
to[out= -30,in=90, looseness=0.9] (-1,0)
to[out= -90,in=60, looseness=0.9] (-1.25,-0.75)
to[out= -60,in=0, looseness=0.9](-1.73,-1.8);
\draw[color=black,scale=0.45,domain=-3.15: 3.15,
smooth,variable=\t,shift={(6.28,-4.5)},rotate=0]plot({2.*sin(\t r)},
{2.*cos(\t r)}) ; 
\draw[scale=0.33,shift={(26.8,-5.5)}] 
(-2,0) 
to[out= 170,in=40, looseness=1] (-2.9,1.2) 
to[out= -140,in=90, looseness=1] (-3.2,0)
(-2,0)
to[out= -70,in=0, looseness=1] (-3,-0.9) 
to[out= -180,in=-90, looseness=1] (-3.2,0)
(-2,0) 
to[out= 50,in=180, looseness=1] (-1.3,0) 
to[out= 60,in=150, looseness=1.5] (-0.75,1) 
(-1.3,0)
to[out= -60,in=-120, looseness=0.9] (-0.5,-0.75)
(-0.75,1)
to[out= -30,in=90, looseness=0.9] (-0.25,0)
to[out= -90,in=60, looseness=0.9] (-0.5,-0.75);
\draw[color=black,scale=0.33,domain=-3.15: 3.15,
smooth,variable=\t,shift={(25,-5.5)},rotate=0]plot({2.*sin(\t r)},
{2.*cos(\t r)}) ; 
\draw[scale=0.33,shift={(31.1,-7)}] 
(-3,0) 
to[out= 120,in=140, looseness=1] (-2.1,1.4) 
to[out= -40,in=90, looseness=1] (0,0)
(-3,0) 
to[out= -120,in=-180, looseness=1] (-1,-1.3) 
to[out= 0,in=-90, looseness=1] (0,0)
(-3,0) 
to[out= 0,in=-120, looseness=1] (-2.6,0.15) 
to[out= 60,in=180, looseness=1] (-2.3,0) 
to[out= 60,in=150, looseness=1.5] (-1.75,1) 
(-2.3,0)
to[out= -60,in=-120, looseness=0.9] (-1.5,-0.75)
(-1.75,1)
to[out= -30,in=90, looseness=0.9] (-1.25,0)
to[out= -90,in=60, looseness=0.9] (-1.5,-0.75);
\draw[color=black,scale=0.33,domain=-3.15: 3.15,
smooth,variable=\t,shift={(29.58,-7)},rotate=0]plot({2.*sin(\t r)},
{2.*cos(\t r)}) ; 
\draw[scale=0.45, shift={(8,-9.72)}]  
(-3.73,0)
to[out= 50,in=180, looseness=1] (-2.8,0) 
to[out= 60,in=150, looseness=1.5] (-1.5,1) 
(-2.8,0)
to[out=-60,in=180, looseness=0.9] (-1.25,-0.75)
(-1.5,1)
to[out= -30,in=90, looseness=0.9] (-1,0)
to[out= -90,in=60, looseness=0.9] (-1.25,-0.75)
to[out= -60,in=150, looseness=0.9](-0.3,-1.3);
\draw[color=black,scale=0.45,domain=-3.15: 3.15,
smooth,variable=\t,shift={(6.28,-9.72)},rotate=0]plot({2.*sin(\t r)},
{2.*cos(\t r)}) ; 
\draw[scale=0.45, shift={(8,-14.94)}] 
 (-3.73,0) 
to[out= 50,in=180, looseness=1] (-2.3,0.7) 
to[out= 60,in=180, looseness=1.5] (-0.45,1.55) 
(-2.3,0.7)
to[out= -60,in=130, looseness=0.9] (-1,-0.3)
to[out= 10,in=100, looseness=0.9](0.1,-0.8)
(-1,-0.3)
to[out=-110,in=50, looseness=0.9](-2.7,-1.7);
\draw[color=black,scale=0.45,domain=-3.15: 3.15,
smooth,variable=\t,shift={(6.28,-14.94)},rotate=0]plot({2.*sin(\t r)},
{2.*cos(\t r)}) ; 
\draw[scale=0.45, shift={(14.65,-9.72)}]  
(-2,0)
to[out= 170,in=40, looseness=1] (-2.9,1.2) 
to[out= -140,in=90, looseness=1] (-3.2,0)
(-2,0) 
to[out= -70,in=0, looseness=1] (-3,-0.9) 
to[out= -180,in=-90, looseness=1] (-3.2,0)
(-2,0) 
to[out= 50,in=180, looseness=1] (-1.3,0) 
to[out= 60,in=150, looseness=1.5] (-0.75,1) 
(-1.3,0)
to[out= -60,in=-120, looseness=0.9] (-0.5,-0.75)
(-0.75,1)
to[out= -30,in=90, looseness=0.9] (0,1)
(0,-1)
to[out= -90,in=60, looseness=0.9] (-0.5,-0.75);
\draw[color=black,scale=0.45,domain=-3.15: 3.15,
smooth,variable=\t,shift={(12.93,-9.72)},rotate=0]plot({2.*sin(\t r)},
{2.*cos(\t r)}) ; 
\draw
(0,0)--(0,-8.05)
(10.5,0)--(10.5,-8.05)
(7.5,0)--(7.5,-8.05)
(4.5,0)--(4.5,-8.05)
(1.5,0)--(1.5,-8.05)
(0,0)--(10.5,0)
(0,-1)--(10.5,-1)
(0,-3.35)--(10.5,-3.35)
(0,-5.7)--(10.5,-5.7)
(0,-8.05)--(10.5,-8.05);
\path[font=\Large]
(4.35,-3.11) node[left]{Theta}
(10.35,-3.16) node[left]{Eyeglasses}
(4.35,-5.46) node[left]{Lens}
(7.35,-5.46) node[left]{Island}
(4.35,-7.81) node[left]{Tree}
(0.75,-2.2) node[above,rotate=90]{$0$ end--points}
(1.25,-2.2) node[above,rotate=90]{on $\partial\Omega$}
(0.75,-4.55) node[above,rotate=90]{$2$ end--points}
(1.25,-4.55) node[above,rotate=90]{on $\partial\Omega$}
(0.75,-6.9) node[above,rotate=90]{$4$ end--points}
(1.25,-6.9) node[above,rotate=90]{on $\partial\Omega$}
(1.75,-0.5) node[right]{$0$ closed curves}
(4.75,-0.5)node[right]{$1$ closed curve}
(7.75,-0.5) node[right] {$2$ closed curves};
\end{tikzpicture}
\end{center}
\begin{caption}
{Networks with two triple junctions.\label{clas}}
\end{caption}
\end{figure}

\begin{dfnz}
Given a network $\mathbb{S}=\bigcup_{i=1}^n\sigma^i([0,1])$
we denote with $L^i$ the length of the $i$--th curve $\sigma^i$, with $L=L^1+\cdots+L^n$ being the global length of $\mathbb{S}$ and with $A^i$ the area enclosed in the $i$--th loop (if present). 
\end{dfnz}

When we say that a network has a loop $\ell$, we mean that there is a Jordan curve $\Gamma$ in $\mathbb{S}$ that encloses
an area $A$. 
In the case of networks with two triple junctions, 
there are only two cases (see Figure~\ref{clas}):
\begin{itemize}
\item the loop $\ell$ is composed by a single curve $\sigma:[0,1]\to\mathbb{R}^2$,
$\sigma(0)=\sigma(1)$ and at this junction we have an angle of $120$ degrees.
The length $L$ of $\ell$ coincides with the length of $\sigma$.
\item the loop $\ell$ is composed by two curves $\sigma^1,\sigma^2:[0,1]\to\mathbb{R}^2$,
that meet each other at their end--points and at both junctions
there is an angle of $120$ degrees.
The length $L$ of $\ell$ is the sum of the lengths of the two curves $\sigma^1$ and $\sigma^2$.
\end{itemize}

Given a network with two triple junctions composed by $n$ curves 
with $l$ end--points $P^1,P^2,\dots,P^l\in\overline{\Omega}$ (if present)
and two
$3$--points $O^1,O^2\in\Omega$, we will denote with
$\sigma^{pi}$ the curves of this network concurring at the
$3$--point $O^p$ (with $p=1,2$), with the index $i$ varying from one 
to three (this is clearly
redundant as some curves coincides, but useful for the notation). 

\begin{dfnz}\label{probdef}  
Given an initial, regular, $C^2$ network $\SS_0$, 
composed by  $n$ curves $\sigma^i:[0,1]\to\overline{\Omega}$, 
with two 3--points $O^1,O^2\in\Omega$ and $l$ end--points $P^1,P^2,\dots,P^l\in\partial\Omega$  (if present)
in a smooth convex, open set $\Omega\subset\mathbb{R}^{2}$, we say that a family of homeomorphic 
networks $\mathbb{S}_t$, described by the
family of time--dependent curves $\gamma^i(\cdot, t)$, is a solution
of the motion by curvature problem with fixed end--points in the time interval $[0,T)$
if the functions $\gamma^i:[0,1]\times[0,T)\to\overline{\Omega}$
are continuous, there holds $\gamma^i(x,0)=\sigma^i(x)$ for every $x\in[0,1]$ and $i\in\{1,2,\dots,n\}$ (initial data), they 
are at least $C^2$ in space and $C^1$ in time in $[0,1]\times(0,T)$ and satisfy
the following system of conditions for every $x\in[0,1]$, $t\in(0,T)$ , $i\in\{1,2,\dots,n\}$,
\begin{equation}\label{problema}
\begin{cases}
\begin{array}{lll}
\gamma_x^i(x,t)\not=0
\quad &&\text{ regularity}\\
\gamma^r(1,t)=P^r\quad&\text{with}\, 0\leq r \leq l\,
\quad &\text{ end--points condition}\\
\sum_{j=1}^3\tau^{pj}(O^p,t)=0\quad&\text{at every 3--point $O^p$}
\quad &\text{ angles of $120$ degrees}\\
\gamma^i_t=k^i\nu^i+\lambda^i\tau^i\quad&\text{for some continuous functions}\, \lambda^i
\quad &\text{ motion by curvature}
\end{array}
\end{cases}
\end{equation}
where we assumed conventionally (possibly reordering the family of
curves and ``inverting'' their parametrization) that the end--point
$P^r$ of the network is given by $\gamma^r(1,t)$.\\
Moreover, in the third equation we abused a little the notation, denoting with $\tau^{pj}(O^p,t)$
the respective unit normal vectors at $O^p$ of the three curves
$\gamma^{pj}(\cdot,t)$ in the family $\{\gamma^i(\cdot,t)\}$ concurring at the
3--point $O^p$.
\end{dfnz}

\begin{rem}
The equation that describes the motion by curvature
\begin{equation}\label{eqev}
\gamma^i_t(x,t)=k^i(x,t)\nu^i(x,t)+\lambda^i(x,t)\tau^i(x,t)
\end{equation}
differs from the classic way of defining the mean curvature flow 
\begin{equation}\label{mcf}
\gamma_t(x,t)=k(x,t)\nu(x,t)=\kappa(x,t)\,,
\end{equation}
for the presence of the tangential term.
In the case of a closed curve one can always pass from 
equation~\eqref{eqev} to equation~\eqref{mcf} by a (time--dependent) reparametrization
of the curve, this is not possible in our case of planar networks with junctions.
Actually, adding a tangential term is necessary to allow the triple junctions to move.
\end{rem}

\begin{rem}
We want to underline the fact that although our approach is parametric, the flow is geometric, that is invariant by rotation and translation (isometries of $\mathbb{R}^2$) and invariant by reparametrization of the curves.
Hence, a unique solution of the Problem~\ref{problema} cannot be expected. Thus, we introduce the definition of \emph{geometric uniqueness}.

\begin{dfnz}
We say that the curvature flow of an initial $C^2$ network $\SS_0=\bigcup_{i=1}^n\sigma^i([0,1])$ is {\em geometrically unique} (in some regularity class), if all the curvature flows (in such class) satisfying Definition~\ref{probdef} can be obtained each other by means of time--depending reparametrizations.\\
To be precise, this means that if $\SS_t$ and $\widetilde{\SS}_t$ are two curvature flows of $\SS_0$, described by some maps $\gamma^i$ and $\widetilde{\gamma}^i$, there exists a family of  maps $\varphi^i:[0,1]\times[0,T)\to[0,1]$ in $C^0([0,1]\times[0,T))\cap C^2([0,1]\times(0,T))$ such that $\varphi^i(0,t)=0$, $\varphi^i(1,t)=1$, $\varphi^i(x,0)=x$, $\varphi_x^i(x,t)\not=0$ and $\widetilde{\gamma}^i(x,t)={\gamma}^i(\varphi^i(x,t),t)$ for every $(x,t)\in[0,1]\times[0,T)$.
\end{dfnz}

It is then clear that if is geometric uniqueness holds, any curvature
flow gives a unique evolved network as a subset of $\R^2$, for every
time $t\in[0,T)$, which is still the same set also if we change the
parametrization of the initial network by the previous discussion.
\end{rem}

\section{Preliminary results}\label{preliminary}

\subsection{Short time existence}
\begin{teo}\label{c2shorttime}
For any initial $C^2$ regular network
$\SS_0=\bigcup_{i=1}^n\sigma^i([0,1])$ there exists a solution $\gamma^i$
of Problem~\eqref{problema} in a maximal time interval $[0,T)$.\\
Such curvature flow $\SS_t=\bigcup_{i=1}^n\gamma^i([0,1],t)$ is a smooth
flow for every time $t>0$, moreover, the unit tangents $\tau^i$ are
continuous in $[0,1]\times[0,T)$, the functions $k(\cdot,t)$ 
converge weakly in $L^2(ds)$ to $k(\cdot,0)$, as $t\to 0$, and the function 
$\int_{\SS_t}k^2\,ds$ is continuous on $[0,T)$.
\end{teo}
\begin{proof}
See \cite[Theorem 6.8]{mannovplusch}.
\end{proof}

\begin{rem}
Because of the lack of maximum principle, due to the presence of the two triple junctions, the geometric uniqueness of the solution $\gamma^i$
in the natural class of flows $C^2$ in space and $C^1$ in time is an open problem.\\
If one considers the higher regularity class of flows $C^{2+2\alpha}$ in space and $C^{1+\alpha}$ in time, with $\alpha\in(0,1/2)$, the geometric uniqueness of the solution $\gamma^i$ is established under the so called 
"compatibility conditions" of order $2$ (see~\cite{bronsard} or~\cite[Section~4]{mannovplusch}, for details) holding for the initial regular network.
\end{rem}

\begin{rem} 
It should be noticed that if the initial curves $\sigma^i$
    are $C^\infty$, the flow $\SS_t$ is smooth till $t=0$ far from
    the 3--points, that is, in any closed rectangle included in
    $(0,1)\times[0,T)$ we can locally reparametrize
    the curves $\gamma^i$ to get a smooth flow up to $t=0$.
    This follows from standard local estimates for the motion by curvature (see~\cite{eckhui2}).
\end{rem}

\begin{rem}
In~\cite{Ilnevsch} a short time existence theorem that does not even require the $120$ degrees angle condition is proved. Indeed, in~\cite[Theorem 1.1]{Ilnevsch} the initial network $\mathbb{S}_0$ is just a connected, planar $C^2$ network with bounded curvature, not necessarily regular, but only with mutually distinct unit tangent vectors at the multi--points.
\end{rem}

The previous theorem says that a flow {\em starts} and it is regular for some time, the next basic theorem describes what happens at the maximal time of smooth existence.
\begin{teo}\label{curvexplod-general} 
If $T<+\infty$ is the maximal time interval of existence of the 
curvature flow $\SS_t$ of an initial regular $C^2$ network given by
the previous theorem, then 
\begin{enumerate} 
\item either the inferior limit of the length of at least one curve of   $\SS_t$ goes to zero when $t\to T$, 
\item or $\limup_{t\to T}\int_{\SS_t}k^2\,ds=+\infty$, hence, the curvature is not bounded as $t\to T$.
\end{enumerate} 
Moreover, if the lengths of the $n$ curves are uniformly positively bounded from below, 
then this superior limit is actually a limit and  
there exists a positive constant $C$ such that 
$$
\int_{{\SS_t}} k^2\,ds \geq \frac{C}{\sqrt{T-t}}\,\,\text{ and }\,\,
\max_{\SS_t}k^2\geq\frac{C}{\sqrt{T-t}}
$$
for every $t\in[0, T)$.
\end{teo}

\begin{proof}
See~\cite[Theorem 3.18]{mannovtor} for the case of a single triple junction  
or~\cite[Theorem 6.7, Corollary 6.10]{mannovplusch} for the general case.
\end{proof}

\subsection{Geometric properties of the flow}
The next proposition describes the evolution of the lengths of the curves of the network and of the areas enclosed in the loops of the network.

\begin{prop}\label{ppp1}
Consider a network with two triple junctions $\mathbb{S}_t=\bigcup_{i=1}^n\gamma^i(x,t)$,
then the evolution equation of its total length is given by
\begin{equation}\label{evolL}
\frac{dL(t)}{dt}=-\int_{\mathbb{S}_t}k^2\,ds\,,
\end{equation}
while the length of a single curve $\gamma^i$ behaves as  
\begin{equation}
\frac{dL^i(t)}{dt}=\lambda^i(1,t)-\lambda^i(0,t)-\int_{\gamma^i(\cdot,t)}(k^i)^2\,ds\,.
\end{equation}
Moreover, the evolution of the area of the region enclosed in a loop is
\begin{equation}\label{evolarea}
\frac{dA^i(t)}{dt}=-2\pi+m\left(\frac{\pi}{3}\right), 
\end{equation}
where $m\in\{1,2\}$ is the number of curves composing the loop.
\end{prop}
\begin{proof}
The evolution of the areas is obtained by a direct application of Gauss--Bonnet Theorem. 
For the proof of the evolution of the lengths see~\cite[Proposition~5.1]{mannovplusch}. 
\end{proof}

\subsection{Huisken's monotonicity formula and the rescaling procedure}\label{rescaling}
Now we introduce a  modified version of Huisken's monotonicity formula (see~\cite{huisk3})
and a dynamical rescaling procedure suitable to our situation.

Let $F : \SS\times [0,T)\to\mathbb{R}^2$ be a curvature flow of a
regular network in its maximal time interval of existence. With a little abuse of notation, we will write $\tau(P^r,t)$ and
$\lambda(P^r,t)$ respectively for the unit tangent vector and the
tangential velocity at the end--point $P^r$ of the curve of the
network, for any $r\in\{1,2,\dots,l\}$.

Let $x_0\in\R^2, t_0 \in (0,\infty)$ and
$\rho_{x_0,t_0}:\R^2\times[0,t_0)$ be the one--dimensional {\em
backward   heat kernel} in $\R^2$ relative to $(x_0,t_0)$, that is, 
$$
\rho_{x_0,t_0}(x,t)=\frac{e^{-\frac{\vert x-x_0\vert^2}{4(t_0-t)}}}{\sqrt{4\pi(t_0-t)}}\,.
$$
We will often write $\rho_{x_0}(x,t)$ to denote $\rho_{x_0,T}(x,t)$
(or $\rho_{x_0}$ to denote $\rho_{x_0,T}$), when $T$ is 
the maximal (singular) time of existence of a smooth curvature flow.

\begin{dfnz}[Gaussian densities]\label{Gaussiandensities}
For every $x_0\in\R^2, t_0 \in (0,\infty)$ we define the {\em Gaussian
  density function} $\Theta_{x_0,t_0}:[0,\min\{t_0,T\})\to\R$ as 
$$
\Theta_{x_0,t_0}(t)=\int_{\SS_t}\rho_{x_0,t_0}(x,t)\,ds
$$
and provided $t_0\leq T$, the {\em limit density function} $\widehat{\Theta}:\R^2\times (0,\infty)\to\R$ as
\[
\widehat{\Theta}(x_0,t_0)=\lim_{t\to t_0}\Theta_{x_0,t_0}(t)\,.
\]
Moreover, we will often write $\Theta_{x_0}(t)$ to denote $\Theta(x_0,T)$ and $\widehat{\Theta}(x_0)$ for $\widehat{\Theta}(x_0,T)$.
\end{dfnz}

The limit $\widehat{\Theta}$ exists and it is finite. 
Moreover, the map $\widehat{\Theta}:\mathbb{R}^2\to\mathbb{R}$ is upper semicontinuous (see~\cite[Proposition~2.12]{MMN13}).

Next, we introduce the rescaling procedure of Huisken in~\cite{huisk3} at the maximal time $T$.\\ 
Fixed $x_0\in\R^2$, 
let $\widetilde{F}_{x_0}:\SS\times [-1/2\log{T},+\infty)\to\R^2$ 
be the map 
$$ \widetilde{F}_{x_0}(p,\tt)
=\frac{F(p,t)-x_0}{\sqrt{2(T-t)}}\qquad \tt(t)
=-\frac{1}{2}\log{(T-t)} $$ 
then, the rescaled networks are given by  
$$ \widetilde{\SS}_{x_0,\tt}=\frac{\SS_t-x_0}{\sqrt{2(T-t)}} $$ 
and they evolve according to the equation  
$$ \frac{\partial\,}{\partial   \tt}\widetilde{F}_{x_0}(p,\tt)
=\widetilde{\underline{v}}(p,\tt)+\widetilde{F}_{x_0}(p,\tt) $$
where  
$$ \widetilde{\underline{v}}(p,\tt)=\sqrt{2(T-t(\tt))}\cdot\underline{v}(p,t(\tt))=
 \widetilde{\underline{k}}+\widetilde{\underline{\lambda}}= 
\widetilde{k}\nu+\widetilde{\lambda}\tau\qquad 
\text{ and }\qquad t(\tt)=T-e^{-2\tt}\,. 
$$ 
Notice that we did not put the sign ``$\;\;\widetilde{}\;\;$'' over the unit tangent and normal, 
since they remain the same after the rescaling.\\ 
We will write $\widetilde{O}^p(\tt)=\widetilde{F}_{x_0}(O^p,\tt)$  
for the 3--points of the rescaled network $\widetilde{\SS}_{x_0,\tt}$ 
and $\widetilde{P}^r(\tt)=\widetilde{F}_{x_0}(P^r,\tt)$ for the end--points, 
when there is no ambiguity on the point $x_0$.\\ 
The rescaled curvature evolves according to the following equation, 
\begin{equation*}\label{evolriscforf} 
{\partial_\tt} \widetilde{k}= \widetilde{k}_{\ss\ss}+\widetilde{k}_\ss\widetilde{\lambda}
+ \widetilde{k}^3 -\widetilde{k} 
\end{equation*} 
which can be obtained by means of the commutation law 
\begin{equation*}\label{commutforf} 
{\partial_\tt}{\partial_\ss}={\partial_\ss}{\partial_\tt} 
+ (\widetilde{k}^2 -\widetilde{\lambda}_\ss-1){\partial_\ss}\,, 
\end{equation*} 
where we denoted with $\ss$ the arclength parameter for $\widetilde{\SS}_{x_0,\tt}$.

By a straightforward computation (see~\cite{huisk3})
we have the following rescaled version of the monotonicity formula.

\begin{prop}[Rescaled monotonicity formula] 
Let $x_0\in\R^2$ and set   
$$ \widetilde{\rho}(x)=e^{-\frac{\vert x\vert^2}{2}} $$ 
For every $\tt\in[-1/2\log{T},+\infty)$ the following identity holds 
\begin{equation*} 
\frac{d\,}{d\tt}\int_{\widetilde{\SS}_{x_0,\tt}} \widetilde{\rho}(x)\,d\ss= 
-\int_{\widetilde{\SS}_{x_0,\tt}}\vert \,\widetilde{\underline{k}}+x^\perp\vert^2\widetilde{\rho}(x)\,d\ss
+\sum_{r=1}^l\Bigl[\Bigl\langle\,{\widetilde{P}^r(\tt)} \,\Bigl\vert\,
{\tau}(P^r,t(\tt))\Bigr\rangle-\widetilde{\lambda}(P^r,\tt)\Bigl]\,\widetilde{\rho}(\widetilde{P}^r(\tt))\label{reseqmonfor} 
\end{equation*} 
where $\widetilde{P}^r(\tt)=\frac{P^r-x_0}{\sqrt{2(T-t(\tt))}}$.\\  
Integrating between $\tt_1$ and $\tt_2$ with  $-1/2\log{T}\leq \tt_1\leq \tt_2<+\infty$ we get 
\begin{align} 
\int_{\tt_1}^{\tt_2}\int_{\widetilde{\SS}_{x_0,\tt}}\vert \,\widetilde{\underline{k}}
+x^\perp\vert^2\widetilde{\rho}(x)\,d\ss\,d\tt= &\, 
\int_{\widetilde{\SS}_{x_0,\tt_1}}\widetilde{\rho}(x)\,d\ss 
-\int_{\widetilde{\SS}_{x_0,\tt_2}}\widetilde{\rho}(x)\,d\ss\label{reseqmonfor-int}\\ &\,
+\sum_{r=1}^l\int_{\tt_1}^{\tt_2}\Bigl[\Bigl\langle\,\widetilde{P}^r(\tt)\,
\Bigl\vert\,{\tau}(P^r,t(\tt))\Bigr\rangle-\widetilde{\lambda}(P^r,\tt)\Bigl]\,\widetilde{\rho}(\widetilde{P}^r(\tt)\,d\tt\,.\nonumber 
\end{align} 
\end{prop}

\begin{lemma}\label{rescstimadib} 
For every $r\in\{1,2,\dots,l\}$ and $x_0\in\R^2$,
the following estimate holds  
\begin{equation*} 
\left\vert\int_{\tt}^{+\infty}\Bigl[\Bigl\langle\,\widetilde{P}^r(\xi)\,
\Bigl\vert\,{\tau}(P^r,t(\xi))\Bigr\rangle-\widetilde{\lambda}(P^r,\xi)\Bigl]\,d\xi\,\right\vert\leq C\,,
\end{equation*}
where $C$ is a constant depending only on the fixed end--points $P^r$.\\
As a consequence, for every point $x_0\in\R^2$, we have
\begin{equation*} 
\lim_{\tt\to  +\infty}\sum_{r=1}^l\int_{\tt}^{+\infty}\Bigl[\Bigl\langle\,\widetilde{P}^r(\xi)\,
\Bigl\vert\,{\tau}(P^r,t(\xi))\Bigr\rangle-\widetilde{\lambda}(P^r,\xi)\Bigl]\,d\xi=0\,. 
\end{equation*} 
\end{lemma}

\subsection{Shrinkers}\label{shri}

\begin{dfnz}\label{shrinkers} A regular $C^2$ network $\SS=\bigcup_{i=1}^n\sigma^i(I^i)$, complete and without end--points, 
is called a {\em regular shrinker} if at every point $x\in\SS$ there holds the {\em shrinkers equation}
\begin{equation}\label{shrinkeq}
\underline{k} + x^\perp=0\,, 
\end{equation}
and if junctions are present, they are only triple junctions forming angles of $120$ degrees.
\end{dfnz}

The name comes from the fact that if $\SS=\bigcup_{i=1}^n\sigma^i(I^i)$ (where $I^i$
is the interval $[0,1]$, $[0,1)$, $(0,1]$ or $(0,1)$) is a shrinker, 
then the evolution given by $\SS_t=\bigcup_{i=1}^n\gamma^i(I^i,t)$, 
where $\gamma^i(x,t)=\sqrt{1-2t}\, \sigma^i(x)$, is a self--similarly shrinking curvature flow
in the time interval $(-\infty,\frac12)$ with $\SS=\SS_{0}$.
Viceversa, if $\SS_t$ is a self--similarly shrinking curvature flow in the maximal time interval $(-\infty,\frac12)$, 
then $\SS_{0}$ is a shrinker.

\begin{dfnz}
A {\em standard triod} is a shrinker triod composed of three halflines from the origin meeting at $120$ degrees.\\
A {\em Brakke spoon} is a shrinker composed by a halfline which intersects a closed curve, 
forming angles of $120$ degrees (first mentioned in~\cite{brakke}).\\
A {\em standard lens} is a shrinker with two triple junctions, it is 
symmetric with respect to two perpendicular axes,
composed by two halflines pointing the origin, posed on a symmetry axis and opposite with respect to the other.
Each halfline intersects two equal curves forming  an angle of $120$ degrees.\\
A {\em fish} is a shrinker with the same topology of the standard lens, but 
symmetric with respect to only one axis.
The two halflines, pointing the origin, intersect two different curves, forming angles of $120$ degrees.
\end{dfnz}
\begin{figure}[H]
\begin{center}
\begin{tikzpicture}[scale=0.6]
\draw[shift={(-10.33,0)}]
(-2,-2) to [out=45, in=-135, looseness=1] (2,2);
\draw[dashed,shift={(-10.33,0)}]
(-2.5,-2.5) to [out=45, in=-135, looseness=1] (-2,-2)
(2,2)to [out=45, in=-135, looseness=1](2.6,2.6);
\draw[shift={(-3.3,0)}]
(0,0) to [out=90, in=-90, looseness=1] (0,2)
(0,0) to [out=210, in=30, looseness=1] (-1.73,-1)
(0,0) to [out=-30, in=150, looseness=1](1.73,-1);
\draw[dashed, shift={(-3.3,0)}]
(0,2) to [out=90, in=-90, looseness=1] (0,3)
(-1.73,-1) to [out=210, in=30, looseness=1] (-2.59,-1.5)
(1.73,-1) to [out=-30, in=150, looseness=1](2.59,-1.5);
\draw[color=black, scale=0.4, shift={(9.875,0)}]
(-3.035,0)to[out= 60,in=180, looseness=1] (2.2,3)
(2.2,3)to[out= 0,in=90, looseness=1] (5.43,0)
(-3.035,0)to[out= -60,in=180, looseness=1] (2.2,-3)
(2.2,-3)to[out= 0,in=-90, looseness=1] (5.43,0);
\draw[color=black, scale=0.4, shift={(9.875,0)}]
(-7,0)to[out= 0,in=180, looseness=1](-3,0);
\draw[color=black, scale=0.4, shift={(9.875,0)}, dashed]
(-9,0)to[out= 0,in=180, looseness=1](-7,0);
\draw[color=black,scale=0.4, shift={(-23,-15)}]
(-3.035,0)to[out= 60,in=180, looseness=1] (2.2,2.8)
(2.2,2.8)to[out= 0,in=120, looseness=1] (7.435,0)
(2.2,-2.7)to[out= 0,in=-120, looseness=1] (7.435,0)
(-3.035,0)to[out= -60,in=180, looseness=1] (2.2,-2.7);
\draw[color=black,scale=0.4, shift={(-23,-15)}]
(-7,0)to[out= 0,in=180, looseness=1](-3.035,0)
(7.435,0)to[out= 0,in=180, looseness=1](11.4,0);
\draw[color=black,dashed,scale=0.4, shift={(-23,-15)}]
(-9,0)to[out= 0,in=180, looseness=1](-7,0)
(11.4,0)to[out= 0,in=180, looseness=1](13.4,0);
\draw[color=black,shift={(2,-6)},scale=1.5]
(-0.47,0)to[out= 20,in=180, looseness=1](1.5,0.65)
(1.5,0.65)to[out= 0,in=90, looseness=1] (2.37,0)
(-0.47,0)to[out= -20,in=180, looseness=1](1.5,-0.65)
(1.5,-0.65)to[out= 0,in=-90, looseness=1] (2.37,0);
\draw[white, very thick,shift={(2,-6)},scale=1.5]
(-0.47,0)--(-.150,-0.13)
(-0.47,0)--(-.150,0.13);
\draw[color=black,shift={(2,-6)},scale=1.5]
(-.150,0.13)to[out= -101,in=90, looseness=1](-.18,0)
(-.18,0)to[out= -90,in=101, looseness=1](-.150,-0.13);
\draw[black,shift={(2,-6)},scale=1.5]
(-.150,0.13)--(-1.13,0.98)
(-.150,-0.13)--(-1.13,-0.98);
\draw[black, dashed,shift={(2,-6)},scale=1.5]
(-1.13,0.98)--(-1.50,1.31)
(-1.13,-0.98)--(-1.50,-1.31);
\fill(-3.3,0) circle (2pt);
\fill(-10.33,0) circle (2pt);
\fill(4,0) circle (2pt);
\fill(-8.33,-6) circle (2pt);
\fill(2,-6) circle (2pt);
\path[font=\small]
(3,-5.9) node[left]{$O$}
(-7.4,-5.8) node[left]{$O$}
(-10.33,0) node[left]{$O$}
(4.9,.2) node[left]{$O$}
(-3.3,.2) node[left]{$O$};
\path[font=\small]
(-10,-3) node[left]{$1$}
(4.4,-3) node[left]{$2.b$}
(-3.4,-3) node[left]{$2.a$}
(3,-9) node[left]{$3.b$}
(-9,-9) node[left]{$3.a$};
\end{tikzpicture}
\end{center}
\begin{caption}{Embedded, regular shrinkers without triple junctions ($1$ line),
with one triple junction ($2.a$ standard triod, $2.b$ Brakke spoon), and with two triple junctions
($3.a$ standard lens, $3.b$ fish).\label{shr}}
\end{caption}
\end{figure}

If we require that these shrinkers are embedded and with multiplicity $1$, they are unique, up to rotations, among the network with the same shape.

We are actually interested in the classification of {\em all} possible complete, embedded, regular shrinkers with at most two triple junctions (and without end--points). The only embedded shrinking curves in $\mathbb{R}^2$ are the lines through the origin and the unit circle, by the work of Abresch-Langer~\cite{ablang1}.
The complete, embedded, connected regular shrinkers with only one triple junction are exactly, up to rotations, the standard triod and the Brakke spoon (see~\cite{chenguo}). 
Now, if a regular shrinker is a tree, all its unbounded curves must be halflines from the origin (see~\cite[Lemma~8.9]{mannovplusch} or use the equation satisfied by the curvature $k_s=k\langle \gamma\,\vert\,\tau\rangle$), which implies that the two triple junctions should coincide with the origin, a contradiction, thus, such a shape is excluded. Then, by an argument of H\"attenschweiler~\cite[Lemma~3.20]{haettenschweiler}, if a regular shrinker contains a region bounded by a single curve, the shrinker must be a Brakke spoon, that is, no other triple junctions can be present. Thus, by looking at the topological shapes in Figure~\ref{clas}, it remains to discuss only two cases: one is the ``lens'' shape and the
other is the shape of the Greek ``theta'' letter (or ``double cell''). It is well known that there exist unique (up to a rotation)
lens--shaped or fish--shaped, complete, embedded, regular shrinkers which are
symmetric with respect to a line through the origin of $\R^2$
(see~\cite{chenguo,schnurerlens}) and it was recently shown in~\cite{balhausman} that it does not exist a theta--shaped shrinker 
(see in Figure~\ref{homtheta} showing how it was supposed to be).
Hence, the ones depicted in Figure~\ref{shr} are actually, up to rotations, the only complete, embedded, regular shrinkers with at most two triple junctions, without end--points.

\begin{figure}[H]
\begin{center}
\begin{tikzpicture}[scale=0.40, rotate=90]
\draw[color=black]
(0,2.4)to[out= 30,in=180, looseness=1] (2.2,3)
(2.2,3)to[out= 0,in=90, looseness=1] (5.43,0)
(0,-2.4)to[out= -30,in=180, looseness=1] (2.2,-3)
(2.2,-3)to[out= 0,in=-90, looseness=1] (5.43,0);
\draw[color=black, rotate=180]
(0,2.4)to[out= 30,in=180, looseness=1] (2.2,3)
(2.2,3)to[out= 0,in=90, looseness=1] (5.43,0)
(0,-2.4)to[out= -30,in=180, looseness=1] (2.2,-3)
(2.2,-3)to[out= 0,in=-90, looseness=1] (5.43,0);
\draw[color=black]
(0,2.4)to[out= 90,in=-90, looseness=0](0,-2.4);
\fill(0,0) circle (3.8pt);
\path[font=\small]
(0,-0.5) node[above]{$O$};
\end{tikzpicture}
\end{center}
\begin{caption}{A theta--shaped shrinker does not exists.\label{homtheta}}
\end{caption}
\end{figure}

If we also consider {\em degenerate} regular shrinkers (see the next section), we also have the union of four halflines
from the origin forming alternate angles of $120$ and $60$ degrees, a configuration which is going to play a key role in the sequel.

\subsection{Degenerate regular networks}\label{degsec}

A ``degenerate network'' is, roughly speaking, a network with one or more ``collapsed'' (degenerate) curves, not visible if one considers it simply as a subset of $\R^2$. To describe its actual ``hidden'', non--collapsed structure we associate a graph to it, where the collapsed curves are present. This is a necessary concept  in order to deal with networks whose curves get shorter and shorter during the flow and ``vanish'' at some time, producing a degenerate network where the associate graph ``remembers'' the structure of the networks before the collapsing time. Moreover, since we want to keep track of the exterior unit tangent vectors of the curves concurring at the triple junctions of the collapsing curves, we also need to ``assign'' some unit vectors to the end--points of the collapsed curves in a way that, in the graph, at every 3--point the sum of such ``assigned'' unit vectors is zero (like for regular networks).

\begin{dfnz}
Consider a couple $(G,\mathbb{S})$ with the following properties:

\begin{itemize}
\item $G=\bigcup_{i=1}^{n}E^i$ is an oriented graph with possible unbounded edges $E^i$, 
such that every vertex has only one or three concurring edges (we call end--points of $G$ the vertices with order one);
\item given a family of $C^1$ curves $\sigma^i:I^i\to\R^2$, 
where $I^i$ is the interval $(0,1)$, $[0,1)$, $(0,1]$ or $[0,1]$,
and orientation preserving homeomorphisms $\varphi^i:E^i\to I^i$, 
then, $\SS=\bigcup_{i=1}^{n}\sigma^{i}(I^i)$
(notice that the interval $(0,1)$ can only appear if it is associated to an unbounded edge $E^i$ without vertices, 
which is clearly a single connected component of $G$);
\item in the case that $I^i$ is $(0,1)$, $[0,1)$ or $(0,1]$, 
the map $\sigma^i$ is a regular $C^1$ curve with unit tangent vector field $\tau^i$;
\item in the case that $I^i=[0,1]$, the map $\sigma^i$ is either a regular $C^1$ curve with unit tangent vector field $\tau^i$, 
or a constant map and in this case it is ``assigned'' also a {\em constant} unit vector $\tau^i:I^i\to\R^2$, 
that we still call unit tangent vector of $\sigma^i$ (we call these maps $\sigma^i$ ``degenerate curves'');
\item for every degenerate curve $\sigma^i:I^i\to\R^2$ with assigned unit vector $\tau^i:I^i\to\R^2$,
we call ``assigned exterior unit tangents'' of the curve $\sigma^i$ at the points $0$ and $1$ of $I^i$, 
respectively the unit vectors $-\tau^i$ and $\tau^i$.
\item the map $\Gamma:G\to\R^2$ given by the union $\Gamma=\bigcup_{i=1}^n(\sigma^i\comp\varphi^i)$ 
is well defined and continuous;
\item for every 3--point of the graph $G$, where the edges $E^i$, $E^j$, $E^k$ concur, 
the exterior unit tangent vectors (real or ``assigned'') at the relative borders of the intervals $I^i$, $I^j$, $I^k$ 
of the concurring curves $\sigma^i$, $\sigma^j$ $\sigma^k$ have zero sum ({\em ``degenerate $120$ degrees condition''}).
\end{itemize}
Then, we call $\SS$ a {\em degenerate regular network}.

If one or several edges $E^i$ of $G$ are mapped under the map $\Gamma:G\to\R^2$ to a single point $p\in\R^2$, we call this
sub--network given by the union $G^\prime$ of such edges $E^i$, the {\em core} of $\SS$ at $p$. 

We call multi--points of the degenerate regular network $\SS$, the images of the vertices of multiplicity three of the graph $G$, by the map $\Gamma$.

We call end--points of the degenerate regular network $\SS$, the images of the vertices of multiplicity one of the graph $G$, by the map $\Gamma$.
\end{dfnz}

\begin{dfnz}
We call {\em degenerate regular shrinker} a degenerate regular network such that every non--collapsed curve satisfies the shrinkers 
equation~\eqref{shrinkeq} .
\end{dfnz}

\section{A geometric quantity}\label{dsuL}
Given the smooth flow $\SS_t=F(\SS,t)$, we take two points $p=F(x,t)$ and
$q=F(y,t)$ belonging to $\SS_t$.
A couple  $\left( p=F(x,t),q=F(y,t)\right)$ is in the
class $\mathfrak{A}$ of the admissible ones if
the segment joining $p$ and $q$ does not intersect the network $\SS_t$ in other points. For any admissible pair $\left( p=F(x,t),q=F(y,t)\right)$ we consider the set 
of the injective curves $\left( \Gamma_{p,q}\right)$ contained in $\SS_t$ connecting $p$ and $q$,
forming with the segment $\overline{pq}$ a Jordan curve.
Thus, it is well defined the area of the open region $\mathcal{A}_{p,q}$ 
enclosed by any Jordan curve constructed in this way and,
for any pair $(p,q)$ we call $A_{p,q}$
the smallest area of such possible regions $\mathcal{A}_{p,q}$.
If $p$ and $q$ are both points of the curves that generates a loop,
we define $\psi(A_{p,q})$ as 
$$
\psi(A_{p,q})=\frac{A}{\pi}\sin\left( \frac{\pi}{A}A_{p,q}\right)\,,
$$
where $A=A(t)$ is the area of the connected component
of $\Omega\setminus\SS_t$ which contains the open segment
joining $p$ and $q$.

We consider the function 
$\Phi_t: \SS\times\SS\to \R\cup\{+\infty\}$ as 
\begin{equation*}
\Phi_t(x,y)= 
\begin{cases}
\frac{\vert p-q\vert^2}{\psi(A_{p,q})}\qquad &\,\text{if}\; x\not=y,\,x,y\; 
\text{are points of a loop},\\
\frac{\vert p-q\vert^2}{A_{p,q}}\qquad &\,\text{if}\; x\not=y,\,x,y\;\text{are 
not both points of a loop},\\
4\sqrt{3}\;\; &\text{  if}\; x\, \text{and} \,y\, \text{coincide with one of the $3-$points}\, O^i \,\text{of}\, \,\SS,\\
+\infty\;\; &\text{ {  if} $x=y\not=O^i$},
\end{cases}
\end{equation*} 
where $p=F(x,t)$ and $q=F(y,t)$.

\begin{rem}
Following the argument of Huisken in~\cite{huisk2},
in the definition of the function $\Phi_t$ we introduce the function $\psi(A_{p,q})$ when
the two points belong to a loop, because we want to maintain the function smooth
also when $A_{p,q}$ is equal to $\frac{A}{2}$.
\end{rem}

In the following, with a little abuse of notation, we consider the function $\Phi_t$ defined on 
$\SS_t\times\SS_t$ and we speak of admissible pair for the couple of points
$(p,q)\in\SS_t\times\SS_t$ instead of $(x,y)\in\SS\times\SS$.

We define $E(t)$ as the infimum of $\Phi_t$ between all admissible couple of points
$p=F(x,t)$ and $q=F(y,t)$:
\begin{equation}\label{equant}
E(t) = \inf_{(p,q)\in\mathfrak{A}}\Phi_t
\end{equation}
for every $t\in[0,T)$.

We call $E(t)$ ``embeddedness measure''.
We notice that similar geometric quantities have already been applied to similar problems 
in~\cite{hamilton3},~\cite{chzh} and~\cite{huisk2}.

The following lemma holds, for its proof in the case of a compact network see~\cite[Theorem~2.1]{chzh}.
\begin{lemma}\label{lemet1}
The infimum of the function $\Phi_t$ between all admissible couples $(p,q)$ is actually a minimum.
Moreover, assuming that $0<E(t)<4\sqrt{3}$, for any minimizing pair 
$(p,q)$ we have $p\ne q$ and
neither $p$ nor $q$ coincides with one of the 3--point $O^i(t)$ of $\SS_t$.
\end{lemma}

Notice that the network $\SS_t$ is embedded {\em if and only if} $E(t)>0$.\\
Moreover, $E(t)\leq 4\sqrt{3}$ always holds, thus when
$E(t)>0$ the two points $(p,q)$ of a minimizing pair can coincide if
and only if $p=q=O^i(t)$.\\
Finally, since the evolution is smooth it is easy to see that the
function $E:[0,T)\to\R$ is locally Lipschitz,
in particular, $\frac{dE(t)}{dt}>0$ exists for almost every time $t\in[0,t)$.

\begin{figure}[H]
\begin{center}
\begin{tikzpicture}[rotate=25, scale=1.25]
\draw[color=black,scale=1,domain=-3.15: 3.15,
smooth,variable=\t,rotate=0]plot({1*sin(\t r)},
{1*cos(\t r)}) ; 
\draw[scale=0.5]
(-2,0) to [out=45, in=-160, looseness=1] (-0.85,0.25)
(-0.85,0.25) to [out=-40, in=150, looseness=1] (0.75,-0.35)
(-0.85,0.25)to [out=80, in=-90, looseness=1] (0,2)
(0.75,-0.35)to [out=30, in=-120, looseness=1](2,0)
(0.75,-0.35)to [out=-90, in=90, looseness=1](0,-2);
\draw[color=black,scale=1,domain=-3.15: 3.15,
smooth,variable=\t,rotate=0,shift={(2,0)}]plot({1*sin(\t r)},
{1*cos(\t r)}) ; 
\draw[scale=0.5]
(6,0) to [out=-135, in=20, looseness=1] (4.85,-0.25)
(4.85,-0.25) to [out=140, in=-30, looseness=1] (3.25,0.35)
(4.85,-0.25)to [out=-100, in=90, looseness=1] (4,-2)
(3.25,0.35)to [out=-150, in=90, looseness=1](2,0)
(3.25,0.35)to [out=90, in=-90, looseness=1](4,2);
\draw[color=black,scale=1,domain=-3.15: 3.15,
smooth,variable=\t,rotate=0,shift={(-2,0)}]plot({1*sin(\t r)},
{1*cos(\t r)}) ; 
\draw[scale=0.5]
(-2,0) to [out=-135, in=40, looseness=1] (-3.15,-0.25)
(-3.15,-0.25) to [out=140, in=-30, looseness=1] (-4.75,0.35)
(-3.15,-0.25)to [out=-100, in=90, looseness=1] (-4,-2)
(-4.75,0.35)to [out=-150, in=60, looseness=1](-6,0)
(-4.75,0.35)to [out=90, in=-90, looseness=1](-4,2);
\draw[color=black,scale=1,domain=-3.15: 3.15,
smooth,variable=\t,rotate=0,shift={(0,2)}]plot({1*sin(\t r)},
{1*cos(\t r)});
\draw[scale=0.5]
(2,4) to [out=-135, in=20, looseness=1] (0.85,3.75)
(0.85,3.75) to [out=140, in=-30, looseness=1] (-0.75,4.35)
(0.85,3.75)to [out=-100, in=90, looseness=1] (0,2)
(-0.75,4.35)to [out=-150, in=60, looseness=1](-2,4)
(-0.75,4.35)to [out=90, in=-90, looseness=1](0,6); 
\draw[color=black,scale=1,domain=-3.15: 3.15,
smooth,variable=\t,rotate=0,shift={(0,-2)}]plot({1*sin(\t r)},
{1*cos(\t r)}) ; 
\draw[scale=0.5]
(2,-4) to [out=-135, in=20, looseness=1] (0.85,-4.25)
(0.85,-4.25) to [out=140, in=-30, looseness=1] (-0.75,-3.65)
(0.85,-4.25)to [out=-100, in=90, looseness=1] (0,-6)
(-0.75,-3.65)to [out=-150, in=60, looseness=1](-2,-4)
(-0.75,-3.65)to [out=90, in=-90, looseness=1](0,-2);
\path[font=\footnotesize,rotate=-25]
(-2.75,-0.4)  node[left]{$\mathbb{H}^1$}
(0.3,-2.77) node[left]{$\mathbb{H}^2$}
(3.35,0.4) node[left]{$\mathbb{H}^3$}
(0.1,2.85) node[left]{$\mathbb{H}^4$}
(-0.35,-0.45)  node[left]{$P^1$}
(0.45,-0.77) node[left]{$P^2$}
(0.9,0.4) node[left]{$P^3$}
(0.15,0.8) node[left]{$P^4$}
(0.1,0.1) node[left]{$O^1$}
(0.55,-0.2) node[left]{$O^2$};
\end{tikzpicture}
\end{center}
\begin{caption}{A tree-shaped network $\mathbb{S}$ with its $\mathbb{H}^i$.}
\end{caption}
\end{figure}

If the network flow $\mathbb{S}_t$ has fixed end--points $\{P^1, P^2,\ldots,P^l\}$
on the boundary of a strictly convex set $\Omega$, 
we consider the flows $\mathbb{H}^i_t$
each obtained as the union of $\mathbb{S}_t$ 
with its reflection $\mathbb{S}^R_i$ with respect to the end--point $P^i$.
We underline that this is still a $C^2$ (or as regular as $\mathbb{S}_t$) curvature flow (as we have a solution till the parabolic boundary, the curvature at the fixed end--points of $\SS_t$ is zero and all the other relations obtained differentiating the evolution equation -- called ``compatibility conditions'', see~\cite[Definition~4.15]{mannovplusch} and the discussion therein -- are satisfied at the end--points) without self--intersections, where $P^i$ is no more an end--point and the
number of triple junctions of ${\mathbb{H}}^i_t$ is exactly twice the number of the ones of $\SS_t$.

We define for the networks $\mathbb{H}^i_t$ the functions $E^i:[0,T)\to\mathbb{R}$, analogous to the function
$E:[0,T)\to\mathbb{R}$ of $\mathbb{S}_t$ and, for every $t\in[0,T)$,
we call $\Pi(t)$ the minimum of the values $E^i(t)$.
The function $\Pi:[0,T)\to\mathbb{R}$ is still a locally Lipschitz 
function (hence, differentiable for almost every time),
clearly satisfying $\Pi(t)\leq E^i(t)\leq E(t)$ for all $t\in[0,T)$.
Moreover, as there are no self--intersections by construction, we have $\Pi(0)>0$. 

If we prove that
$\Pi(t)\geq C>0$ for all $t\in[0,T)$, for some constant $C\in\mathbb{R}$,
then we can conclude that also $E(t)\geq C>0$  for all $t\in[0,T)$.

\begin{teo}\label{dlteo} 
Let $\Omega$ be a open, bounded, strictly convex subset of $\mathbb{R}^2$.
Let  $\mathbb{S}_0$ be an initial network
with two triple junctions, 
and let $\mathbb{S}_t$ be a smooth 
evolution by curvature of $\mathbb{S}_0$
defined in a maximal time interval $[0,T)$.

Then,  there exists  a constant $C>0$ depending only on $\SS_0$ 
such that $E(t)\ge C>0$ for every $t\in[0,T)$.
In particular, the networks $\SS_t$ remain embedded during the flow.
\end{teo}

To prove this theorem we first show the next proposition and lemma.

\begin{prop}\label{lemet2} 
For every $t\in[0,T)$ such that
\begin{itemize}
\item $0<E(t)<1/4$,
\item for at least one minimizing pair $(p,q)$ of $\Phi_t$, neither $p$ nor $q$ coincides with one of the fixed end--points $P^i$,
\end{itemize}
if the derivative $\frac{dE(t)}{dt}$ exists, then it is positive.
\end{prop}

By simplicity, we consider in detail only the cases shown in Figure~\ref{disuelle}.
The computations in the other situations are analogous. 

\begin{figure}[H]
\begin{center}
\begin{tikzpicture}[scale=0.85]
\draw[color=black,scale=1,domain=-3.15: 3.15,
smooth,variable=\t,shift={(-1,0)},rotate=0]plot({3.25*sin(\t r)},
{2.5*cos(\t r)}) ;
\filldraw[fill=black!10!white,shift={(-8,0)}]
(6,0.5) 
to[out=-165,in=60,  looseness=1] (5.19,0)--(5.19,0)
to[out=-60,in=150,  looseness=1] (7.62,-0.81)--(7.62,-0.81)  
to[out=30,in=-167.7, looseness=1.5] (8.483,-0.59) -- (6,0.5) ;
\draw[color=black!20!white]
(0,1) to[out=-90, in=90, looseness=1] (0,-0.67)
(0.15,0.98) to[out=-90, in=90, looseness=1] (0.15,-0.61)
(0.3,0.92) to[out=-90, in=90, looseness=1] (0.3,-0.585)
(0.45,0.87) to[out=-90, in=90, looseness=1] (0.45,-0.565)
(0.6,0.8) to[out=-90, in=90, looseness=1] (0.6,-0.55)
(0.75,0.7) to[out=-90, in=90, looseness=1] (0.75,-0.52)
(0.9,0.65) to[out=-90, in=90, looseness=1] (0.9,-0.5)
(1.05,0.57) to[out=-90, in=90, looseness=1] (1.05,-0.45)
(1.2,0.52) to[out=-90, in=90, looseness=1] (1.2,-0.4)
(1.35,0.4) to[out=-90, in=90, looseness=1] (1.35,-0.265)
(-0.15,1.025) to[out=-90, in=90, looseness=1] (-0.15,-0.68)
(-0.3,1.035) to[out=-90, in=90, looseness=1] (-0.3,0.15)
(-0.3,0) to[out=-90, in=90, looseness=1] (-0.3,-0.75)
(-0.45,1.02) to[out=-90, in=90, looseness=1] (-0.45,0.35)
(-0.45,0.1) to[out=-90, in=90, looseness=1] (-0.45,-0.74)
(-0.6,0.98) to[out=-90, in=90, looseness=1] (-0.6,0.1)
(-0.6,0) to[out=-90, in=90, looseness=1] (-0.6,-0.67)
(-0.75,0.93) to[out=-90, in=90, looseness=1] (-0.75,-0.63)
(-0.9,0.84) to[out=-90, in=90, looseness=1] (-0.9,-0.6)
(-1.05,0.75) to[out=-90, in=90, looseness=1] (-1.05,-0.573)
(-1.2,0.655) to[out=-90, in=90, looseness=1] (-1.2,-0.56)
(-1.35,0.595) to[out=-90, in=90, looseness=1] (-1.35,-0.545)
(-1.5,0.559) to[out=-90, in=90, looseness=1] (-1.5,-0.535)
(-1.65,0.526) to[out=-90, in=90, looseness=1] (-1.65,-0.527)
(-1.8,0.51) to[out=-90, in=90, looseness=1] (-1.8,-0.517)
(-1.95,0.47) to[out=-90, in=90, looseness=1] (-1.95,0)
(-1.95,-0.2) to[out=-90, in=90, looseness=1] (-1.95,-0.5)
(-2.1,0.445) to[out=-90, in=90, looseness=1] (-2.1,0.18)
(-2.1,-0.17) to[out=-90, in=90, looseness=1] (-2.1,-0.47)
(-2.25,0.405) to[out=-90, in=90, looseness=1] (-2.25,0)
(-2.25,-0.1) to[out=-90, in=90, looseness=1] (-2.25,-0.43)
(-2.4,0.34) to[out=-90, in=90, looseness=1] (-2.4,-0.36)
(-2.55,0.24) to[out=-90, in=90, looseness=1] (-2.55,-0.255)
(-2.7,0.12) to[out=-90, in=90, looseness=1] (-2.7,-0.12);
\draw[color=white,shift={(-8,0)}]
(7.62,-0.81)  
to[out=30,in=-167.7, looseness=1.5] (8.483,-0.59);
\draw[shift={(-8,0)}]
 (3.75,0) node[left] {$P^1$} to[out=30,in=110, looseness=1] (4.37,0) 
to[out= -50,in=180, looseness=1] (5.19,0)node[above] {$O^1$}
to[out=60, in=-145, looseness=1] (7,0.81)
to[out=35,in=150,  looseness=1] (8.62,0.81) 
to[out=-30,in=90, looseness=1.5] (9.44,0)
(7.62,-0.81) 
to[out=30,in=-90,  looseness=1](9.44,0)
(7.62,-0.81) 
to[out=150,in=-60,  looseness=1] (5.19,0)
(7.62,-0.81) 
to[out=-90,in=60,  looseness=1] (9.1,-1.9)node[below]{$P^2$};
\draw
 (0.438,-0.58) -- (-1.98,0.49); 
\path[font= \footnotesize] 
(-2,0.3) node[below] {$\mathcal{A}_{p,q}$};
\path[font= \Large]
(-3.75,-1.8) node[below] {$\Omega$};
\path[font=\large]
(-0.45,0.53) node[below] {$A$}
(-0.3,-1)  node[left]{$O^2$}
(-2.03,1) node[below] {$p$}
(0.7,-0.52) node[below] {$q$};
\end{tikzpicture}
\end{center}
\begin{caption}{The situation considered in the computations of Proposition~\ref{lemet2} .}\label{disuelle}
\end{caption}
\end{figure}

\begin{proof} 
Let $0<E(t)<1/4$ and let $(p,q)$ be a minimizing pair for $\Phi_t$ 
such that the two points are both distinct
from the end--points $P^i$. 
We choose a value  $\eps>0$ smaller than the ``geodesic'' distances of $p$ and $q$ from
the 3--points of $\SS_t$ and between them.\\
Possibly taking a smaller $\eps>0$, we fix an arclength
coordinate $s\in (-\eps,\eps)$ and a local parametrization $p(s)$ of
the curve containing in a neighborhood of $p=p(0)$, with the 
same orientation of the original one.
Let $\eta(s)=\vert p(s)-q\vert$, since  
\begin{equation*}
E(t)=\min_{s\in(-\varepsilon,\varepsilon)}\frac{\eta^2(s)}{\psi(A_{p(s),q})}=
\frac{\eta^2(0)}{\psi(A_{p,q})}\,,
\end{equation*}
if we differentiate in $s$, we obtain 
\begin{equation}\label{eqdsul}
\frac{d\eta^2(0)}{ds}\psi(A_{p(0),q})=
\frac{d\psi(A_{p(0),q})}{ds}\eta^2(0)\,.
\end{equation}
We underline that we are considering the function $\psi$ because we are doing all
the computation for the case shown in Figure~\ref{disuelle}, where there is a loop.
For a network without loops the computations are simpler:
instead of formula~\eqref{eqdsul}, one has 
$$
\frac{d\eta^2(0)}{ds}A_{p(0),q}=\frac{dA_{p(0),q}}{ds}\eta^2(0)\,,
$$ see~\cite[Page~281]{mannovtor}, for instance.

As the intersection of the segment $\overline{pq}$ with the network is transversal, 
we have an angle $\alpha(p)\in(0,\pi)$ 
determined by the unit tangent $\tau(p)$ and the vector $q-p$.\\
We compute
\begin{align*} 
\left.\frac{{ d} \eta^2(s)}{{ d} s}\right\vert_{s=0}
 &=\,
-2 \langle \tau(p)\,\vert\, q-p\rangle = -2 \vert
p-q\vert \cos\alpha(p)\\ 
\left.\frac{{ d} A(s)}{{ d} s}\right\vert_{s=0} 
&=\, 0\\
\left.\frac{{ d} A_{p(s),q}}{{ d} s}\right\vert_{s=0} 
&=\, 
\frac 12 \vert \tau(p) \wedge (q-p)\vert = 
\frac 12 \langle \nu(p)\,\vert\, q-p\rangle = 
\frac 12  \vert p-q\vert \sin\alpha(p)\\
\left.\frac{{ d} \psi(A_{p(s),q})}{{ d} s} \right\vert_{s=0} 
&=\, 
\frac{{ d} A_{p,q}}{{ d} s}\cos\left(\frac{\pi}{A}A_{p,q}\right)\\
\, &=\,
\frac 12  \vert p-q\vert \sin\alpha(p)\cos\left(\frac{\pi}{A}A_{p,q}\right)\,. \\
\end{align*}
Putting these derivatives in equation~\eqref{eqdsul} and
recalling that $\eta^2(0)/\psi(A_{p,q})=E(t)$,  we get
\begin{align}\label{topolino2}
\cot\alpha(p) 
&\,= -\frac{\vert p-q\vert^2}{4\psi(A_{p,q})}\cos\left(\frac{\pi}{A}A_{p,q}\right)\nonumber\\
&\,= -\frac{E(t)}{4}\cos\left(\frac{\pi}{A}A_{p,q}\right)\,.
\end{align}
Since $0<E(t)<\frac{1}{4}<4(2-\sqrt{3})$,
we have $\sqrt{3}-2<\cot\alpha(p)<0$ which
implies
\begin{equation}\label{alpha}
\frac{\pi}{2} <\alpha(p) < \frac {7\pi}{12} \,.
\end{equation}
The same argument clearly holds for the point $q$, hence 
defining $\alpha(q)\in(0,\pi)$ to be the angle determined by the
unit tangent $\tau(q)$ and the vector $p-q$, by equation~\eqref{topolino2}
it follows that $\alpha(p)=\alpha(q)$ and we simply write $\alpha$
for both.\\
We consider now a different variation, moving at the same time the
points $p$ and $q$, in such a way that $\frac{dp(s)}{ds}=\tau(p(s))$ and 
$\frac{dq(s)}{ds}=\tau(q(s))$.\\
As above, letting $\eta(s)=\vert p(s)-q(s)\vert$, by minimality we have
\begin{align}\label{eqdersec}
\frac{{ d} \eta^2(0)}{{ d} s}\left.\psi(A_{p(s),q(s)})\right\vert_{s=0}
&=
\left( \left.\frac{{ d} \psi(A_{p(s),q(s)})}{{ d} s}\right\vert_{s=0}\right) \eta^2(0) \;\;\text{ { and} }\;\;\nonumber\\
\frac{{ d}^2 \eta^2(0)}{{ d} s^2}\left.\psi(A_{p(s),q(s)})\right\vert_{s=0}
&\ge 
\left( \left.\frac{{ d}^2\psi( A_{p(s),q(s)})}{{ d} s^2}\right\vert_{s=0}\right) \eta^2(0)\,.
\end{align}
Computing as before,
\begin{align*}
\left.\frac{{ d} \eta^2(s)}{{ d} s}\right\vert_{s=0}
&=\,2 \langle p-q \,\vert\,\tau(p)-\tau(q)\rangle = -4\vert p-q\vert\cos\alpha \\
\left.\frac{{ d} A_{p(s),q(s)}}{{ d} s}\right\vert_{s=0} 
&=\, -\frac12\langle p-q\,\vert\,\nu(p)+\nu(q)\rangle
=+\vert p-q\vert\sin\alpha\\
\left.\frac{{ d}^2 \eta^2(s)}{{ d} s^2}\right\vert_{s=0} 
&=\,2 \langle \tau(p)-\tau(q) \,\vert\,\tau(p)-\tau(q)\rangle +
2 \langle p-q \,\vert\, k(p)\nu(p) -k(q)\nu(q)\rangle\\
&=\, 2\vert \tau(p)-\tau(q)\vert^2+
2 \langle p-q \,\vert\, k(p)\nu(p) -k(q)\nu(q)\rangle\\
&=\, 8\cos^2\alpha+
2 \langle p-q \,\vert\, k(p)\nu(p) -k(q)\nu(q)\rangle\\
\left.\frac{{ d}^2 A_{p(s),q(s)}}{{ d} s^2}\right\vert_{s=0} 
&=\,-\frac12\langle \tau(p)-\tau(q)\,\vert\,\nu(p)+\nu(q)\rangle
+\frac12\langle p-q\,\vert\,k(p)\tau(p)+k(q)\tau(q)\rangle\\
&=\,-\frac12\langle\tau(p)\,\vert\,\nu(q)\rangle
+\frac12\langle \tau(q)\,\vert\,\nu(p)\rangle\\
&+\frac12\langle p-q\,\vert\,k(p)\tau(p)+k(q)\tau(q)\rangle\\
&=\,-2\sin\alpha\cos\alpha
-1/2\vert p-q\vert(k(p)-k(q))\cos\alpha\\
\left.\frac{{ d}^2 \psi(A_{p(s),q(s)})}{{ d} s^2}\right\vert_{s=0} 
&=\,\left.\frac{{ d}}{{ d} s}\left\lbrace
\frac{{ d} A_{p(s),q(s)}}{{ d} s}
\cos\left(\frac{\pi}{A}A_{p(s),q(s)}\right) 
\right\rbrace\right\vert_{s=0}\\
&=\, (-2\sin\alpha\cos\alpha-\frac12\vert p-q\vert(k(p)-k(q))\cos\alpha)
\cos\left( \frac{\pi}{A}A_{p,q}\right) \\
&\,-\frac{\pi}{A}\vert p-q\vert^2 \sin^2\alpha\sin\left(\frac{\pi}{A}A_{p,q}\right)\,. \\
\end{align*} 
Substituting the last two relations in the second inequality 
of~\eqref{eqdersec}, we get 
\begin{align*}
&(8\cos^2\alpha+\, 
2\langle p-q \,\vert\, k(p)\nu(p) -k(q)\nu(q)\rangle)\psi(A_{p,q})\\
&\geq\, \vert p-q\vert^2
\left\lbrace 
(-2\sin\alpha\cos\alpha-\frac12\vert p-q\vert(k(p)-k(q))\cos\alpha)
\cos\left( \frac{\pi}{A}A_{p,q}\right)\right.\\
&\,\left.-\frac{\pi}{A}\vert p-q\vert^2 \sin^2\alpha\sin\left(\frac{\pi}{A}A_{p,q}\right) 
\right\rbrace \,,
\end{align*}
hence, keeping in mind that 
$\tan\alpha=
\frac{-4}{E(t)\cos\left(\frac{\pi}{A}A_{p(s),q(s)} \right) }$, we obtain
\begin{align}\label{eqfin}
&2\psi(A_{p,q})\langle p-q \,\vert\, \,k(p)\nu(p) -k(q)\nu(q)\rangle\nonumber\\
&\,+1/2\vert p-q\vert^3(k(p)-k(q))\cos\alpha 
\cos\left(\frac{\pi}{A}A_{p,q}\right)\nonumber\\
&\,\geq -2\sin\alpha\cos\alpha\vert p-q\vert^2\cos\left(\frac{\pi}{A}A_{p,q}\right)\nonumber\\
&\,-8\psi(A_{p,q})\cos^2\alpha
+\vert p-q\vert^4 \sin^2\alpha
\left[-\frac{\pi}{A}\sin\left(\frac{\pi}{A}A_{p,q}\right) \right]\nonumber\\
&=\,-2\psi(A_{p,q})\cos^2\alpha
\left(\tan\alpha\frac{\vert p-q\vert^2}{\psi(A_{p,q})}
\cos\left(\frac{\pi}{A}A_{p,q}\right)+ 4\right)\nonumber\\
&\,+\vert p-q\vert^4\sin^2\alpha
\left[ -\frac{\pi}{A}\sin\left(\frac{\pi}{A}A_{p,q}\right) \right] \nonumber\\
&=\,+\vert p-q\vert^4\sin^2\alpha
\left[ -\frac{\pi}{A}\sin\left(\frac{\pi}{A}A_{p,q}\right) \right]\,. \nonumber\\
\end{align}
We now compute the derivative
$\frac{dE(t)}{dt}$ by means of the Hamilton's trick (see~\cite{hamilton2}),
that is,
$$
\frac{dE(t)}{dt} = \frac{\partial}{\partial t}\Phi_t(\overline{p},\overline{q})\,,
$$
for {\em any} minimizing pair $(\overline{p},\overline{q})$ for $\Phi_t$.
In particular, $\frac{dE(t)}{dt}=\frac{\partial}{\partial t}\Phi_t(p,q)$ and,
we recall,  $\frac{\vert p-q\vert^2}{\psi(A_{p,q})}=E(t)$.\\
Notice that by minimality of the pair $(p,q)$, we are free to choose the 
``motion'' of the points $p(\tau)$, $q(\tau)$ ``inside'' the networks $\Gamma_{\tau}$
in computing such partial derivative, that is
$$
\frac{dE(t)}{dt}=\frac{\partial}{\partial t}\Phi_t(p,q)=\frac{d}{d\tau}\Phi_t(p(\tau),q(\tau))\Big\vert_{\tau=t}\,.
$$
Since locally the networks are moving by curvature and we know that
neither  $p$ nor $q$ coincides with the 3--point, 
we can find $\varepsilon>0$ and two  smooth curves $p(\tau), q(\tau)\in\Gamma_\tau$  for
every $\tau\in(t-\varepsilon,t+\varepsilon)$ such that
\begin{align*}
p(t) &=\, p \qquad \text{ and }\qquad \frac{dp(\tau)}{d\tau} = k(p(\tau), \tau)
~\nu(p(\tau),\tau)\,, \\
q(t) &=\, q \qquad \text{ and }\qquad \frac{dq(\tau)}{d\tau}= k(q(\tau),\tau)
~\nu(q(\tau),\tau)\,. 
\end{align*}
Then,
\begin{equation}\label{eqderE}
\left.\frac{dE(t)}{dt} =
\frac{\partial}{\partial t}\Phi_{t}(p,q)
=\frac{1}{\left[ \psi(A_{p,q})\right] ^2}
\left(\psi(A_{p,q})
\frac{d\vert p(\tau)-q(\tau)\vert^2}{d\tau} - 
\vert p-q\vert^2 \frac{d\psi(A_{p(\tau),q(\tau)})}{d\tau}\right)\right\vert_{\tau=t}\,.
\end{equation}
With a straightforward computation we get the following equalities, 
\begin{align*}
\left.\frac{{ d} \vert p(\tau)-q(\tau)\vert^2}{{ d} \tau} \right\vert_{\tau=t}
&\,= 2 \langle p-q\,\vert\, k(p)\nu(p) -k(q)\nu(q)\rangle\\
\left.\frac{{ d} (A(\tau))}{{ d} \tau}\right\vert_{\tau=t} 
&\,= -\frac{4\pi}{3}\\
\left.\frac{{ d} A_{p(\tau),q(\tau)}}{{ d} \tau}\right\vert_{\tau=t} 
&\,=\int_{\Gamma_{p,q}} \langle\underline{k}(s)\,\vert\nu_{\xi_{p,q}}\rangle\,ds
+ \frac12\vert p-q\vert\langle {\nu_{[p,q]}}\,\vert\, k(p)\nu(p)+k(q)\nu(q)\rangle\\
&\,= 2\alpha -\frac{4\pi}{3}-\frac 12 \vert p-q\vert(k(p)-k(q))\cos\alpha\\
\left.\frac{{ d} \psi(A_{p(\tau),q(\tau)})}{{ d} \tau} \right\vert_{\tau=t}
&\,=-\frac{4\pi}{3}\left[\frac{1}{\pi}\sin\left(\frac{\pi}{A}A_{p,q}\right)
-\frac{A_{p,q}}{A}\cos\left(\frac{\pi}{A}A_{p,q} \right)\right] \\
&\,+\left(2\alpha -\frac{4\pi}{3}-\frac 12 \vert p-q\vert(k(p)-k(q))\cos\alpha\right)
\cos\left(\frac{\pi}{A}A_{p,q} \right)\\
\end{align*}
where we wrote $\nu_{\xi_{p,q}}$ and $\nu_{[p,q]}$ for the exterior
unit normals to the region $A_{p,q}$, respectively at the
points of the geodesic $\xi_{p,q}$ and of the segment $\overline{pq}$.\\
We remind that in general  $\frac{{ d} (A(t))}{{ d} t}=-2\pi+m\left(\frac{\pi}{3}\right) $
where $m$ is the number of triple junctions of the loop
(see~\eqref{evolarea}), we obtain  $\frac{{ d} (A(t))}{{ d} t}=-\frac{4\pi}{3}$
because we are referring to Figure~\ref{disuelle}, where there is a loop with exactly two triple junctions.\\
Substituting these derivatives in equation~\eqref{eqderE} we get
\begin{align*}
\frac{dE(t_0)}{dt} 
=&\,\frac{2\langle p-q\,\vert\, k(p)\nu(p) -k(q)\nu(q)\rangle}{\psi(A_{p,q})}\\
-&\frac{\vert p-q\vert^2}{(\psi(A_{p,q}))^2}
\left\lbrace 
-\frac{4\pi}{3}\left[\frac{1}{\pi}\sin\left(\frac{\pi}{A}A_{p,q}\right)
-\frac{A_{p,q}}{A}\cos\left(\frac{\pi}{A}A_{p,q}\right)\right]
\right. \\
&\left. +\left(2\alpha -\frac{4\pi}{3}-\frac 12 \vert p-q\vert(k(p)-k(q))\cos\alpha\right)
\cos\left(\frac{\pi}{A}A_{p,q} \right)
 \right\rbrace \\
\end{align*}
and, by equation~\eqref{eqfin}, \\
\begin{align*}
\frac{dE(t_0)}{dt} 
\geq 
&\,-\frac{\vert p-q\vert^2}{(\psi(A_{p,q}))^2}
\left\lbrace 
-\frac{4}{3}\sin\left(\frac{\pi}{A}A_{p,q}\right)
+\frac{4\pi}{3}\frac{A_{p,q}}{A}\cos\left(\frac{\pi}{A}A_{p,q}\right) \right.\\
&\,\left. +\left(2\alpha -\frac{4\pi}{3}\right)
\cos\left(\frac{\pi}{A}A_{p,q} \right)+\frac{\pi}{A}\vert p-q \vert^2 \sin^2(\alpha) 
\sin\left(\frac{\pi}{A}A_{p,q}\right)
\right\rbrace\,.\\
\end{align*}
It remains to prove that the quantity 
\begin{align*}
&\frac{4}{3}\sin\left(\frac{\pi}{A}A_{p,q}\right)
-\frac{4\pi}{3}\frac{A_{p,q}}{A}\cos\left(\frac{\pi}{A}A_{p,q} \right) 
 +\left(\frac{4\pi}{3}-2\alpha\right)
\cos\left(\frac{\pi}{A}A_{p,q} \right)\\
&-\frac{\pi}{A}\vert p-q \vert^2 \sin^2(\alpha) 
\sin\left(\frac{\pi}{A}A_{p,q}\right)\\
\end{align*}
is positive.\\
As $E(t)=\frac{\vert p-q\vert^2}{\psi(A_{p,q})}=\frac{\vert p-q\vert^2}{\frac{A}{\pi}\sin(\frac{\pi}{A}A_{p,q})}$
we can write
\begin{align*}
&\frac{4}{3}\sin\left(\frac{\pi}{A}A_{p,q}\right)
-\frac{4\pi}{3}\frac{A_{p,q}}{A}\cos\left(\frac{\pi}{A}A_{p,q} \right) 
 +\left(\frac{4\pi}{3}-2\alpha\right)
\cos\left(\frac{\pi}{A}A_{p,q} \right)\\
&-\frac{\pi}{A}\vert p-q \vert^2 \sin^2(\alpha) 
\sin\left(\frac{\pi}{A}A_{p,q}\right)\\
&= \frac{4}{3}\sin\left(\frac{\pi}{A}A_{p,q}\right)
-\frac{4\pi}{3}\frac{A_{p,q}}{A}\cos\left(\frac{\pi}{A}A_{p,q} \right) 
 +\left(\frac{4\pi}{3}-2\alpha\right)
\cos\left(\frac{\pi}{A}A_{p,q} \right)\\
&-E(t)\sin^2(\alpha) 
\sin^2\left(\frac{\pi}{A}A_{p,q}\right)
\end{align*}
We notice that using~\eqref{alpha}, we can evaluate the sign of $\frac{4\pi}{3}-2\alpha$.\\
We conclude the estimate diving it in two cases related to the value of $\frac{A_{p,q}}{A}$.\\
If $0\leq \frac{A_{p,q}}{A}\leq \frac 13$,
we have
\begin{align*}
\frac{dE(t_0)}{dt} 
&\geq 
\,  
\frac{4}{3}\sin\left(\frac{\pi}{A}A_{p,q}\right)
-\frac{4\pi}{3}\frac{A_{p,q}}{A}\cos\left(\frac{\pi}{A}A_{p,q} \right) \\
& +\left(\frac{4\pi}{3}-2\alpha\right)
\cos\left(\frac{\pi}{A}A_{p,q} \right)-E(t)\sin^2(\alpha) 
\sin^2\left(\frac{\pi}{A}A_{p,q}\right)\\
&\, \geq 
\left(\frac{4\pi}{3}-2\alpha\right)
\cos\left(\frac{\pi}{A}A_{p,q} \right)-E(t)\sin^2(\alpha) 
\sin^2\left(\frac{\pi}{A}A_{p,q}\right)\\
&\, \geq
\left(\frac{\pi}{6}\right)
\cos\left(\frac{\pi}{3} \right)-E(t) 
\sin^2\left(\frac{\pi}{3}\right)>0\,.
\end{align*}
If $\frac 13\leq \frac{A_{p,q}}{A}\leq \frac 12$,
we get
\begin{align*}
\frac{dE(t_0)}{dt} 
&\geq 
\,  
\frac{4}{3}\sin\left(\frac{\pi}{A}A_{p,q}\right)
-\frac{4\pi}{3}\frac{A_{p,q}}{A}\cos\left(\frac{\pi}{A}A_{p,q}\right)\\ 
&+\left(\frac{4\pi}{3}-2\alpha\right)
\cos\left(\frac{\pi}{A}A_{p,q} \right)-E(t)\sin^2(\alpha) 
\sin^2\left(\frac{\pi}{A}A_{p,q}\right)\\
&\, \geq 
\frac{4}{3}\sin\left(\frac{\pi}{A}A_{p,q}\right)
-\frac{4\pi}{3}\frac{A_{p,q}}{A}\cos\left(\frac{\pi}{A}A_{p,q} \right) 
 -E(t)\sin^2(\alpha) 
\sin^2\left(\frac{\pi}{A}A_{p,q}\right)\\
&\, \geq
\frac{4}{3} \left( 
\sin\left(\frac{\pi}{3}\right)
-\frac{\pi}{3}\cos\left(\frac{\pi}{3}\right) 
 \right) 
 -E(t)>0\,.
\end{align*}
Hence, we have proved that, for every $t$ in an interval such that $0<E(t)>\frac14$
and such that the derivative $\frac{dE(t)}{dt}$ exists, $\frac{dE(t)}{dt}>0$
and this suffices to prove the statement.
\end{proof}

\begin{lemma}[Lemma~14.8 in~\cite{mannovplusch}]\label{trinoncoll}
Let $\Omega$ be a open, bounded, strictly convex subset of $\mathbb{R}^2$.
Let  $\mathbb{S}_0$ be an initial regular network
with two triple junctions, 
and let $\mathbb{S}_t$ be the 
evolution by curvature of $\mathbb{S}_0$
defined in a maximal time interval $[0,T)$.
Then, there cannot be a sequence of times $t_j\to T$, such that,
along such sequence, the two triple junctions converge to the same end--point of the network.
\end{lemma}

\begin{rem} Actually, the hypothesis of strict convexity of $\Omega$
can be weakened asking that $\Omega$ is convex and that there are not three aligned end--points of the network on $\partial\Omega$.
\end{rem}

\begin{proof}[Proof of Theorem~\ref{dlteo}]
If $\SS_t$ has not end--points,
the conclusion follows immediately from Proposition~\ref{lemet2}.
Hence, we assume that $\mathbb{S}_t$ has two or four end--points
(in the first case there is a loop, in the second $\mathbb{S}_t$ is a tree),
which are the only possibilities.
Let $t\in[0,T)$ a time such that $0<\Pi(t)<1/4$ and $\Pi$ and all the embeddedness measures $E^i$,
associated to the networks $\mathbb{H}^i_t$ are differentiable at $t$ 
(this clearly holds for almost every time).

Let $E^i(t)=\Pi(t)<1/4$ and $E^i(t)$ is realized by a pair of points $p$ and $q$ in $\mathbb{H}^i_t$,
we separate the analysis in the following case:
\begin{itemize}
\item If the point $p$ and $q$ of the minimizing pair are both end--points of $\mathbb{H}^i$, by construction
$\vert p-q\vert\geq \varepsilon>0$.
Moreover,  the area enclosed in the Jordan curve formed by the segment $\overline{pq}$ and 
by the geodesic curve $\left( \Gamma_{p,q}\right)$ can be 
uniformly bounded by above by a constant $\widetilde{C}$, for instance, 
the area of a ball containing all the networks $\mathbb{H}_t^i$.
Since $\varepsilon>0$ and $\widetilde{C}$ depend only on $\Omega$ 
and on the structure of the initial network $\mathbb{S}_0$
(more precisely on the position of the end--points on the boundary of $\Omega$, that stay fixed during the evolution
and that do not coincide), 
the ratio $\frac{\vert p-q\vert^2}{\psi(A_{p,q})}$ (or $\frac{\vert p-q\vert^2}{A_{p,q}}$,
if $p,q$ do not belong to a loop)
is greater of equal than some constant $C_\varepsilon=\frac{\varepsilon^2}{\widetilde{C}}>0$
hence the same holds for $\Pi(t)$.
\item If one point is internal and the other is an end--point of $\mathbb{H}^i_t$, 
we consider the following two situations.
If one of the two point $p$ and $q$ is in $\mathbb{S}_t\subset\mathbb{H}^i_t$ and the other is
in the reflection $\mathbb{S}^{R_i}_t$, then, we obtain by construction,
a uniform bound from below on $\Pi(t)$ as in the case in which 
$p$ and $q$ are both boundary points of $\mathbb{H}^i_t$.
Otherwise, if $p$ and $q$ are both in $\mathbb{S}_t$ and one of them coincides with $P^j$ with $j\neq i$,
either the other point coincides with $P^i$ and we have again a uniform bound from below on $\Pi(t)$,
as before, or both $p$ and $q$ are points of $\mathbb{H}_t^j$ both not coinciding with its end--points and 
$E^j(t)=E^i(t)=\Pi(t)<1/4$, so we can apply the argument at the next point.
\item
If $p$ and $q$ are both ``inside'' ${\mathbb{H}}^i_t$,
and if the geodesic curve $\Gamma_{pq}$ contains at most two $3$--points,
by Hamilton's trick (see~\cite{hamilton2} or~\cite[Lemma~2.1.3]{Manlib}), 
we have $\frac{d\Pi(t)}{dt}=\frac{dE^i(t)}{dt}$ and, by Proposition~\ref{lemet2}, 
$\frac{dE^i(t)}{dt}>0$, hence $\frac{d\Pi(t)}{dt}>0$.
If instead the geodesic curve $\Gamma_{p.q}$ contains more than two 3--points, we want to show that there exists a uniform positive constant $\varepsilon$
such that $\vert p-q\vert\geq\varepsilon>0$, which implies a uniform positive estimate from below on $E^i(t)$, as above. This will conclude the proof.\\
Assume by contradiction that such a bound is not possible, then, for a sequence of times $t_j\to T$, the Euclidean distance between the two points $p_j$ and $q_j$ of the associated minimizing pair of $\Phi_{t_j}$ goes to zero, as $j\to\infty$, and this can happen only if $p_i,q_i\to P^i$. It follows, by the maximum principle that the two 3--points $O^1(t)$ and $O^2(t)$ converge to $P^i$ on some sequence of times $t_k\to T$ (possibly different by $t_j$), which is forbidden by Lemma~\ref{trinoncoll} and we are done.
\end{itemize}
\end{proof}

As the quantity $E$  is dilation and translation invariant, the following property holds.

\begin{cor}\label{m12trips}
If $\Omega$ is strictly convex and the evolving network $\SS_t$ has at most two triple junctions, every $C^1\loc$--limit of rescalings of networks of the flow 
is embedded and has multiplicity one.
\end{cor}

In other words, this corollary says that the so--called {\em multiplicity--one conjecture} holds for networks with at most two triple junctions, see the discussion in Section~14 of~\cite{mannovplusch}.

\section{Analysis of singularities}\label{longtime}

In this section we first analyze the possible blow--up at a singular time of the evolution of a network with two triple junctions of general topological type, then we discuss in details 
the specific networks, case by case.

\subsection{Limit of rescaling procedure}
\begin{prop}\label{resclimit}
Let $\SS_t=\bigcup_{i=1}^n\gamma^i([0,1],t)$ be a $C^{2,1}$ curvature flow
of networks with two triple junctions in a smooth, strictly convex, bounded open set $\Omega\subset\R^2$.
Then,
for every $x_0\in\R^2$ and for every subset $\mathcal I$ of  $[-1/2\log
T,+\infty)$ with infinite Lebesgue measure, 
there exists a sequence of rescaled times
$\tt_j\to+\infty$, with $\tt_j\in{\mathcal I}$, such that the sequence
of rescaled networks $\widetilde{\SS}_{x_0,\tt_{j}}$
converges in $C^{1,\alpha}\loc\cap W^{2,2}\loc$, for any $\alpha \in (0,1/2)$,
 to a (possibly empty) limit degenerate regular shrinker $\widetilde\SS_\infty$.
 
Moreover, we have
\begin{equation}\label{gggg2}
\lim_{j\to\infty}\frac{1}{\sqrt{2\pi}}\int_{\widetilde{\SS}_{x_0,\tt_j}}
\widetilde{\rho}\,d\sigma=\frac{1}{\sqrt{2\pi}}
\int_{\widetilde\SS_\infty}\widetilde{\rho}\,d{\sigma}=\widehat{\Theta}(x_0)\,.
\end{equation}
\end{prop}

\begin{proof}
See~\cite[Proposition~2.19]{MMN13} and~\cite[Proposition~8.20]{mannovplusch}.
\end{proof}

An important fact is that all the possible limits $\widetilde\SS_\infty$ are embedded network with multiplicity one, by Corollary~\ref{m12trips} in the previous section.

\begin{prop}\label{possiblelimit}
If the rescaling point $x_0$ belongs to $\Omega$, then the blow--up limit network $\widetilde\SS_\infty$ (if not empty) is one of the following:
\begin{itemize}
\item a straight line through the origin;
\item a standard triod centered at the origin;
\item a Brakke spoon;
\item four halflines from the origin forming angles in pair of $120/60$ degrees;
\item a standard lens;
\item a fish.
\end{itemize}
If the rescaling point $x_0$ is a fixed end--point of the evolving network
(on the  boundary of $\Omega$), then the blow--up limit network $\widetilde\SS_\infty$ (if not empty) is one of the following:
\begin{itemize}
\item a halfline from the origin;
\item two halflines from the origin forming an angle of $120$ degrees.
\end{itemize}
\end{prop}
\begin{proof}
The limit (possibly degenerate) network $\widetilde\SS_\infty$ has to satisfy the shrinkers equation $k_\infty+x^\perp=0$
for all $x\in\widetilde\SS_\infty$ (see the proof of Proposition~\ref{resclimit}).

If we assume that $\widetilde\SS_\infty$ is a degenerate regular shrinkers, that is, a core is present, since there are only two 3--points, 
the only possibility is that a single curve (connecting the two triple junctions or a triple junction with an end--point, by Lemma~\ref{trinoncoll}) ``collapses'' in the limit forming 
such a core of $\widetilde\SS_\infty$,  which then must be composed by four halflines from the origin forming 
angles in pair of $120/60$ degrees, if $x_0\in\Omega$, or by 
two halflines from the origin forming an angle of $120$ degrees, when $x_0\in\partial\Omega$.

If $\widetilde\SS_\infty$ is not degenerate and the curvature $k_\infty$ is constantly zero, the network is composed only by 
halflines or straight lines. Then, the possible flat regular shrinkers are either a straight line through the origin or a standard triod, if $x_0\in\Omega$, 
or a halfline, if $x_0\in\partial\Omega$.\\
If instead the curvature is not constantly zero and the network $\widetilde\SS_\infty$ is not degenerate, by the classification of regular shrinkers with two triple junctions that we discussed in Section~\ref{shri}, we can only have either Brakke spoon, or the standard lens or the fish.
In all these three cases the center of the homothety is inside
the enclosed region, hence $x_0$ cannot be an end--point on the boundary of $\Omega$.
\end{proof}

We define the set of reachable points of the flow by
$$
R=\left\lbrace x\in \mathbb{R}^2 \vert\; \text{there exist}\, p_i\in \mathbb{S} \;
\text{and} \,t_i\nearrow T \; \text{such that}
\lim_{i\longrightarrow\infty}F(p_i,t_i)=x \right\rbrace\,. 
$$

Such a set is not empty and compact,
if a point $x_0\notin R$, it means that the flow is definitely far from $x_0$, on the other hand, it can be shown (see~\cite[Lemma~10.4]{mannovplusch}) that if 
$x_0\in R$, for every $t\in [0,T)$ the closed ball of radius $\sqrt{2(T-t)}$ and center $x_0$ 
intersects $\mathbb{S}_t$. Hence, we have a ``dichotomy'' when we consider the blow--up around points of $\overline{\Omega}$:
\begin{itemize}
\item the limit of any sequence of rescaled networks is not empty and we are rescaling around a point in $R$; 
\item the blow--up limit is empty.
\end{itemize}

We remind that,  thanks to Theorem~\ref{curvexplod-general},
if $T<+\infty$ is the maximal time of existence of a smooth flow, then,
at least one of the following two possibilities happens:
\begin{itemize}
\item the length of one (or more) curve of the network goes to zero;
\item the curvature is unbounded as $t\to T$.
\end{itemize}

\subsection{Limit networks with hypothesis on the length of the curves}

We are now ready to analyse the behavior of the flow at a singular time. We start assuming that, as $t\to T$, no curve is collapsing, or more in general, the hypothesis in the following proposition.

\begin{prop}\label{resclength}
Consider a reachable point for the flow $x_0$, the sequence of rescaled networks $\widetilde{\SS}_{x_0,\tt_{j}}$ 
of Proposition~\ref{resclimit} and its
$C^{1}\loc$--limit $\widetilde\SS_\infty$.
If we assume that all the lengths $L^i(t)$ of the curves of the networks satisfy
\begin{equation}\label{Lbasso}
\lim_{t\to T}\frac{L^i(t)}{\sqrt{T-t}}=+\infty\,.
\end{equation}
If the rescaling point belongs to $\Omega$, then $\widetilde\SS_\infty$ is one of the following:
\begin{itemize}
\item a straight line through the origin;
\item a standard triod centered at the origin.
\end{itemize}
If the rescaling point is a fixed end--point of the evolving network
(on the  boundary of $\Omega$), then $\widetilde\SS_\infty$ is:
\begin{itemize}
\item a halfline from the origin.
\end{itemize}
\end{prop}
\begin{proof}
See~\cite[Proposition~8.28]{mannovplusch}. 
\end{proof}

Then, the following theorem implies that the curvature must be bounded.

\begin{teo}\label{regularity}
Let $\mathbb{S}_t$ be a smooth flow in the maximal time interval $[0,T)$ for the initial network $\mathbb{S}_0$.
Let $x_0$ be a reachable point for the flow such that 
the $C^{1}\loc$--limit $\widetilde{\SS}_\infty$ of
the sequence of rescaled networks $\widetilde{\SS}_{x_0,\tt_{j}}$ is:
\begin{itemize}
\item a straight line trough the origin;
\item a halfline from the origin;
\item a standard triod.
\end{itemize}
Then the curvature of the evolving network is uniformly bounded for $t\in[0,T)$ in a ball around the point $x_0$.
\end{teo}
\begin{proof}
By means of White's local regularity theorem in~\cite{white1}, if $\widetilde{\SS}_\infty$ is a straight line or a halfline from the origin (when the blow--up is around a point $x_0\in\partial\Omega$), the curvature is bounded and the same happens if $\widetilde{\SS}_\infty$ is a standard triod centered at the origin (see~\cite{Ilnevsch,MMN13,mannovplusch}). 
\end{proof}

\begin{rem}\label{regrem} Notice that this theorem holds also without the hypothesis of boundedness of the curvature.
\end{rem}

If then we suppose that the lengths of all the curves of the network with two triple junctions stay 
bounded away from zero by a uniform constant $L>0$, during the smooth evolution of $\mathbb{S}_t$
in a maximal time interval of existence $[0,T)$ with $T<+\infty$, by means of Proposition~\ref{resclength} and Theorem~\ref{regularity}, we have a contradiction with Theorem~\ref{curvexplod-general}. Hence, we have the following conclusion.

\begin{prop}\label{global}
Consider a smooth evolution $\mathbb{S}_t$ of a network with two triple junctions
in its maximal time interval of existence $[0,T)$.
If we suppose that the length of all the curves of the network stays
bounded away from zero by a constant $L>0$ during the evolution,
then $T=+\infty$, in other words, the evolution is global in time.
\end{prop}

Therefore, if the maximal time of existence is $T$ is finite, the inferior limit 
of the length of at least one curve, as $t\to T$,  must be zero. Then, we separate the analysis in two cases (see also the beginning of Section~10 in~\cite{mannovplusch}):
\begin{itemize}
\item the curvature is uniformly bounded along the flow and the length of at least one curve of the network goes to zero, when $t\to T$ (Proposition~\ref{ppp1} implies that if the curvature is bounded the length of every curve has a limit, as $t\to T$);
\item the curvature is unbounded and the length of at least one curve of the network is not positively bounded away from zero, as $t\to T$.
\end{itemize}

\subsection{Limit networks with vanishing curves and with bounded curvature}

Assume that the length of at least one curve goes to zero, as $t\to T$, while the curvature of the evolving network is uniformly bounded.
 
\begin{prop}\label{collapse}
If $\SS_t=\bigcup_{i=1}^n\gamma^i([0,1],t)$ is the curvature flow of a regular network with two triple junctions and with fixed end--points, in a
maximal time interval $[0,T)$. If the curvature is uniformly bounded along the flow, then, the networks $\SS_t$ converge in $C^1\loc$,
as $t\to T$, to a degenerate regular network $\SS_T=\bigcup_{i=1}^n\widetilde{\gamma}^i_T([0,1])$.\\
Moreover, every vertex of $\SS_T$ is either a regular triple junction, an end--point of $\SS_t$, or 
\begin{itemize}
\item a $4$--point where the four concurring curves have opposite unit tangents in pairs and form 
angles of $120/60$ degrees between them
(collapse of the curve joining the two triple junctions of $\SS_t$);
\item a $2$--point at an end--point of the network $\SS_t$ where the
  two concurring curves form an angle of $120$
degrees among them (collapse of the curve joining a triple junction to such end--point of $\SS_t$).
\end{itemize}
\end{prop}
\begin{proof}
See~\cite[Proposition~10.11]{mannovplusch} and~\cite[Proposition~10.14]{mannovplusch}. 
\end{proof}

\begin{rem}
We underline that to obtain the previous characterization of 
the degenerate regular network  $\SS_T$ it is necessary that the multiplicity--one conjecture holds.
Indeed, this restricts the admissible cores to a single curve and excludes all other more complex situations, see the following figure.
\end{rem}
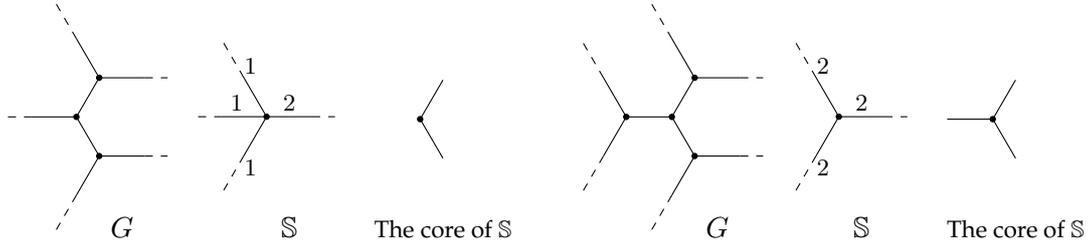
\begin{figure}[H]
\begin{center}
\begin{tikzpicture}[scale=0.30]
\draw[color=black]
(-1,0)to[out= 0,in=180, looseness=1](1,0)
(2,1.73)to[out= 0,in=180, looseness=1](4,1.73)
(2,-1.73)to[out= 0,in=180, looseness=1](4,-1.73)
(2,1.73)to[out= 120,in=-60, looseness=1] (1,3.46)
(2,-1.73)to[out= -120,in=60, looseness=1] (1,-3.46)
(1,0)to[out= 60,in=-120, looseness=1] (2,1.73)
(1,0)to[out=-60,in=120, looseness=1] (2,-1.73);
\draw[color=black,dashed]
(-2,0)to[out= 0,in=180, looseness=1](-1,0)
(4,1.73)to[out= 0,in=180, looseness=1](5,1.73)
(4,-1.73)to[out= 0,in=180, looseness=1](5,-1.73)
(1,3.46)to[out= 120,in=-60, looseness=1] (0,5.19)
(1,-3.46)to[out= -120,in=60, looseness=1] (0,-5.19);
\fill(1,0) circle (4pt);
\fill(2,1.73) circle (4pt);
\fill(2,-1.73) circle (4pt);
\path[font=](3,-4) node[below] {$G$};
\end{tikzpicture}\quad
\begin{tikzpicture}[scale=0.30]
\draw[color=black]
(0,1.73)to[out= 0,in=180, looseness=1](2,1.73)
(2,1.73)to[out= 0,in=180, looseness=1](4,1.73)
(2,1.73)to[out= 120,in=-60, looseness=1] (1,3.46)
(2,1.73)to[out= -120,in=60, looseness=1] (1,0);
\draw[color=black,dashed]
(-1,1.73)to[out= 0,in=180, looseness=1](0,1.73)
(4,1.73)to[out= 0,in=180, looseness=1](5,1.73)
(1,3.46)to[out= 120,in=-60, looseness=1] (0,5.19)
(1,0)to[out= -120,in=60, looseness=1] (0,-1.73);
\fill(2,1.73) circle (4pt);
\path[font=\footnotesize](0.7,1.53) node[above] {$1$};
\path[font=\footnotesize](3,1.53) node[above] {$2$};
\path[font=\footnotesize](1.3,4.76) node[below] {$1$};
\path[font=\footnotesize](1.3,-1.3) node[above] {$1$};
\path[font=](3,-2.27) node[below] {$\mathbb{S}$};
\end{tikzpicture}\quad
\begin{tikzpicture}[scale=0.30]
\draw[color=black]
(1,0)to[out= 60,in=-120, looseness=1] (2,1.73)
(1,0)to[out=-60,in=120, looseness=1] (2,-1.73);
\fill(1,0) circle (4pt);
\path[font=\footnotesize]
(2,-4) node[below] {The core of $\mathbb{S}$};
\end{tikzpicture}\qquad
\begin{tikzpicture}[scale=0.30]
\draw[color=black]
(-1,0)to[out= 120,in=-60, looseness=1] (-2,1.73)
(-1,0)to[out= -120,in=60, looseness=1] (-2,-1.73)
(-1,0)to[out= 0,in=180, looseness=1](1,0)
(2,1.73)to[out= 0,in=180, looseness=1](4,1.73)
(2,-1.73)to[out= 0,in=180, looseness=1](4,-1.73)
(2,1.73)to[out= 120,in=-60, looseness=1] (1,3.46)
(2,-1.73)to[out= -120,in=60, looseness=1] (1,-3.46)
(1,0)to[out= 60,in=-120, looseness=1] (2,1.73)
(1,0)to[out=-60,in=120, looseness=1] (2,-1.73);
\draw[color=black,dashed]
 (-2,1.73)to[out= 120,in=-60, looseness=1](-3,3.46)
 (-2,-1.73)to[out= -120,in=60, looseness=1](-3,-3.46)
(4,1.73)to[out= 0,in=180, looseness=1](5,1.73)
(4,-1.73)to[out= 0,in=180, looseness=1](5,-1.73)
(1,3.46)to[out= 120,in=-60, looseness=1] (0,5.19)
(1,-3.46)to[out= -120,in=60, looseness=1] (0,-5.19);
\fill(-1,0)circle (4pt);
\fill(1,0) circle (4pt);
\fill(2,1.73) circle (4pt);
\fill(2,-1.73) circle (4pt);
\path[font=](3,-4) node[below] {$G$};
\end{tikzpicture}\quad
\begin{tikzpicture}[scale=0.30]
\draw[color=black]
(2,1.73)to[out= 0,in=180, looseness=1](4,1.73)
(2,1.73)to[out= 120,in=-60, looseness=1] (1,3.46)
(2,1.73)to[out= -120,in=60, looseness=1] (1,0);
\draw[color=black,dashed]
(4,1.73)to[out= 0,in=180, looseness=1](5,1.73)
(1,3.46)to[out= 120,in=-60, looseness=1] (0,5.19)
(1,0)to[out= -120,in=60, looseness=1] (0,-1.73);
\fill(2,1.73) circle (4pt);
\path[font=](3,-2.27) node[below] {$\mathbb{S}$};
\path[font=\footnotesize](3,1.53) node[above] {$2$};
\path[font=\footnotesize](1.3,4.76) node[below] {$2$};
\path[font=\footnotesize](1.3,-1.3) node[above] {$2$};
\end{tikzpicture}\quad
\begin{tikzpicture}[scale=0.30]
\draw[color=black]
(1,0)to[out= 60,in=-120, looseness=1] (2,1.73)
(1,0)to[out=-60,in=120, looseness=1] (2,-1.73)
(1,0)to[out=180,in=0, looseness=1] (-1,0);
\path[font=\footnotesize]
(2,-4) node[below] {The core of $\SS$};
\fill(1,0) circle (4pt);
\end{tikzpicture}\qquad
\end{center}
\caption{Two examples of limit degenerate regular networks if the multiplicity--one conjecture does not hold.}
\end{figure}

\subsection{Limit networks with vanishing curves without bounded curvature}

This situation is the most delicate, we start excluding the possibility that the curvature is unbounded and the blow--up limit $\widetilde{\SS}_\infty$ has zero curvature (such limit must be among the ones listed in Proposition~\ref{possiblelimit})

By Theorem~\ref{regularity}, which holds without the hypothesis of boundedness of the curvature, if $\widetilde{\SS}_\infty$ is a straight line or a halfline from the origin (when the blow--up is around a point $x_0\in\partial\Omega$), the curvature is bounded and the same happens if $\widetilde{\SS}_\infty$ is a standard triod centered at the origin. We now exclude also the case that $\widetilde{\SS}_\infty$ is composed of four halflines from the origin forming angles in pair of $120/60$ degrees 
or two halflines from the origin forming an angle of $120$ degrees (when $x_0\in\partial\Omega$).

\begin{prop}\label{cross}
Let $\mathbb{S}_t$ be a smooth flow in the maximal time interval $[0,T)$ 
for the initial network with two triple junctions $\mathbb{S}_0$.
Let $x_0$ be a reachable point for the flow 
and $t_j\to T$ a sequence of times such that the associate rescaled networks $\widetilde{\SS}_{x_0,\tt_{j}}$
(as in Proposition~\ref{resclimit}) converge, as $j\to\infty$,
in $C^{1,\alpha}\loc\cap W^{2,2}\loc$, for any $\alpha\in (0,1/2)$ to:
\begin{itemize}
\item a limit degenerate regular shrinker $\widetilde{\SS}_\infty$ 
composed by four concurring halflines with opposite unit tangent vectors in pairs, forming 
angles of $120/60$ degrees between them ($x_0\in\Omega$);
\item a limit degenerate regular shrinker $\widetilde{\SS}_\infty$ 
composed by two concurring halflines forming an angle of $120$
degrees between them ($x_0\in\partial\Omega$).
\end{itemize}
Then, 
$$
\vert k(x,t)\vert \leq C <+\infty\,,
$$
for all $x$ in a neighborhood of $x_0$ and $t\in[0,T)$.
\end{prop}

\begin{proof}
First we restrict to the case in which the sequence of rescaled networks
$\widetilde{\SS}_{x_0,\tt_{j}}$ converges, as $j\to\infty$, in
$C^{1,\alpha}\loc\cap W^{2,2}\loc$, for any $\alpha \in (0,1/2)$, to
$\widetilde{\SS}_\infty$ limit degenerate regular shrinker 
composed by four concurring halflines with opposite unit tangent vectors in pairs, forming 
angles of $120/60$ degrees between them.
By arguing as in~\cite[Theorem~2.4]{MMN13} or~\cite[Lemma~9.1]{mannovplusch} (keeping into account of~\cite[Lemma~8.18]{mannovplusch} and the second point of Remark~8.21 in the same paper), we can assume
that for $R>0$ large enough there exists $j_0\in\NN$, such the flow
$\SS_t$ has equibounded curvature, no 3--points and an uniform bound from below on the lengths
of the four curves in the annulus $B_{3R\sqrt{2(T-t_j)}}(x_0)\setminus
  B_{R\sqrt{2(T-t_j)}}(x_0)$, for every $t\in[t_j,T)$ and $j\geq j_0$.
We can thus introduce four ``artificial'' moving boundary points
$P^r(t)\in\SS_t$ with $\vert P^r(t)-x_0\vert=2R\sqrt{2(T-t_j)}$, with $r\in\{1,
2, 3, 4\}$ and $t\in [t_j,T)$ such that there 
exist uniform (in time) constants $C_j$, for every $j\in\NN$, such that 
\begin{equation}\label{endsmooth}
\vert\partial_s^jk(P^r,t)\vert+\vert\partial_s^j\lambda(P^r,t)\vert\leq C_j\,,
\end{equation}
for every $t\in[0,T)$ and $r\in{1,2,\dots,l}$.
As the sequence of rescaled networks $\widetilde{\SS}_{x_0,\tt_{j}}$
converges, as $j\to\infty$, in $W^{2,2}\loc$, to a limit network
$\widetilde{\SS}_\infty$ with zero curvature, we have
$$
\lim_{j\to\infty}\Vert
\widetilde{k}\Vert_{L^2(B_{3R}(0)\cap\widetilde{\SS}_{x_0,\tt_{j}})}=0\,,\qquad
\text{ that is, }\qquad 
\int_{B_{3R}(0)\cap\widetilde{\SS}_{x_0,\tt_{j}}}\widetilde{k}^2\,d\sigma\leq\varepsilon_j\,,
$$
for a sequence $\varepsilon_j\to 0$ as $j\to\infty$.
Rewriting this condition for the non--rescaled networks, we have
\begin{equation}\label{notrescaled}
\int_{B_{3R\sqrt{2(T-t_j)}}(x_0)\cap\mathbb{S}_{t_j}} k^2\,ds\leq \frac{\varepsilon_j}{\sqrt{2(T-t_j)}}\,.
\end{equation}

Then, applying~\cite[Lemma~10.23]{mannovplusch}  to the flow of networks $\SS_t$
in the ball $B_{2R\sqrt{2(T-t_j)}}(x_0)$  in the time interval
 $[t_j,T)$, we have that $\Vert k\Vert_{L^2(B_{2R\sqrt{2(T-t_j)}}(x_0)\cap\SS_t)}$ is uniformly bounded, up to time 
$$
T_j=t_j+\min\,\Bigl\{ T, 1\big/
8C\,\bigl(\Vert k
  \Vert^2_{L^2(B_{2R\sqrt{2(T-t_j)}}(x_0)\cap\SS_{t_j})}+1\bigr)^2\Bigr\}\,.
$$
We want to see that actually $T_j>T$ definitely, hence, $\Vert k
  \Vert_{L^2(B_{2R}(x_0)\cap\SS_t)}$ is uniformly bounded for
    $t\in[0,T)$. If this is not true, we have
\begin{align*}
T_j=&\,t_j+\frac{1}{8C\,\bigl(\Vert k  \Vert^2_{L^2(B_{2R\sqrt{2(T-t_j)}}(x_0)\cap\SS_{t_j})}+1\bigr)^2}\\
\geq&\,t_j+\frac{1}{8C\,\bigl(\eps_j/\sqrt{2(T-t_j)}+1\bigr)^2}\\
=&\,t_j+\frac{2(T-t_j)}{8C\,\bigl(\eps_j+\sqrt{2(T-t_j)}\,\bigr)^2}\\
=&\,T+(2(T-t_j))\biggl(\frac{2}{8C\,\bigl(\eps_j+\sqrt{2(T-t_j)}\,\bigr)^2}-1\biggr)\,,
\end{align*}
which is clearly definitely larger than $T$, as $\eps_j\to0$, when
$j\to\infty$.\\
Choosing then $j_1\geq j_0$ large enough, since $\Vert
k\Vert_{L^2(B_{2R\sqrt{2(T-t_{j_1})}}(x_0)\cap\SS_t)}$ is 
uniformly bounded for all times $[t_{j_1},T)$ and 
the length of the four curves that connect the junctions with the ``artificial'' boundary points $P^r(t)$
are bounded below by a uniform constant,
then, by~\cite[Lemma~10.24]{mannovplusch}, the quantity 
$\Vert k_s\Vert_{L^2(B_{2R\sqrt{2(T-t_{j_1})}}(x_0)\cap\SS_t)}$
is uniformly bounded on $[0,T)$.
Moreover, thanks to~\cite[Lemma~10.22]{mannovplusch}, 
in the ball $B_{2R\sqrt{2(T-t_j)}}(x_0)$ we have the uniform in time inequality for $\mathbb{S}_t$ 
\begin{equation}\label{inkappa}
\|k\|_{L^\infty}^2 \le   4C+10\|k\|_{L^2} \|k_s\|_{L^2}\,,
\end{equation}
that is $\|k\|_{L^\infty}$ is bounded  for every $t\in[0,T)$.

If the limit degenerate regular shrinker $\widetilde{\SS}_\infty$ is
composed by two concurring halflines forming an angle of $120$
degrees between them, then we symmetrize $\mathbb{S}_t$ with respect to $x_0\in\partial\Omega$
and we can argue as before for the union of $\mathbb{S}_t$ with the reflected network.
\end{proof}

\begin{proof}[Proof of Proposition~\ref{kscoppia}] By the previous discussion and this proposition, considering the list of possible blow--up limits $\widetilde{\mathbb{S}}_\infty$ given by Proposition~\ref{possiblelimit}, when the curvature is unbounded, the only blow--up limits (up to rotations) are the Brakke spoon, the standard lens and the fish, which is the statement of Proposition~\ref{kscoppia}. Moreover, it also follows that locally around every end--point the curvature stays bounded.
\end{proof}

We now instead describe a case, common to most of the networks with two triple junctions,
in which surely both the curvature is unbounded and at least one length goes to zero.

\begin{prop}\label{loop}
Let $\mathbb{S}_0$ be a network with two triple junctions and with a loop $\ell$ of length $L$, enclosing a region of area $A$ and let $\mathbb{S}_t$ be a smooth evolution by curvature of such network in the maximal time interval $[0,T)$. Then, $T$ is finite and if $\lim_{t\to T}L(t)=0$, there holds $\lim_{t\to T}\int_{\mathbb{S}_t}k^2\,ds=+\infty$.
\end{prop}
\begin{proof}
If a loop is present, by the classification of topological structures of the networks with two triple junctions, it must be composed 
of $m$ curves, with $m\leq 2$, hence, integrating in time equation~\eqref{evolarea}, we have
$$
A(t)-A(0)=\left(-2\pi+m\left(\frac{\pi}{3}\right)\right)t\,,
$$
therefore, $T\leq\frac{3A(0)}{(6-m)\pi}$, otherwise a region of the network collapses before the maximal time, which is impossible.

If $L(t)\to0$ as $t\to T$, also the area $A(t)$ of the region enclosed in the loop must go to zero and $T=\frac{3A(0)}{(6-m)\pi}$. Then, 
combining equation~\eqref{evolarea} and H\"{o}lder inequality, one gets
$$
\Big\vert -2\pi+m\left(\frac{\pi}{3}\right)\Big\vert=\Big\vert \frac{dA(t)}{dt}\Big\vert =\Big\vert \int_{\ell_t} k\,ds\Big\vert
\leq \left(L(t)\right)^\frac12\left(\int_{\ell_t} k^2\,ds \right)^\frac12\,,
$$
hence, 
$$
\int_{\mathbb{S}_t} k^2\,ds\geq\int_{\ell_t} k^2\,ds\geq \frac{\left(6-m\right)^2\pi^2}{9L(t)}\,.
$$
Then clearly, when $t\to T$, as $L(t)\to 0$, the $L^2$--norm of the curvature goes to infinity.
\end{proof}

\subsection{Proof of the main result}

\begin{proof}[Proof of Theorem~\ref{main}]
Let $\mathbb{S}_t$ be a smooth evolution by curvature of a network with two triple junctions and (possibly) fixed end--points on $\partial\Omega$, with $\Omega$ regular, 
open and strictly convex subset of $\mathbb{R}^2$, in a maximal time interval $[0,T)$.

If a loop is present, by Proposition~\ref{loop}, the maximal time of smooth existence $T$ is finite. If such time $T$ is smaller than the ``natural'' time that the loop shrinks (depending on the number of curves composing the loop, as in Proposition~\ref{loop}), the network is locally a tree, uniformly for $t\in[0,T)$. Hence, every blow--up limit at any point $x_0\in\overline{\Omega}$ cannot contain loops, then Proposition~\ref{possiblelimit} shows that it must have zero curvature, thus, by Theorem~\ref{regularity} and Proposition~\ref{cross} the curvature of $\SS_t$ is uniformly bounded along the flow and (see Proposition~\ref{collapse}) converges, as $t\to T$, to a degenerate regular network $\SS_T$ with vertices that are either a regular triple junction, an end--point, or 
\begin{itemize}
\item a $4$--point where the four concurring curves have opposite unit tangents in pairs and form 
angles of $120/60$ degrees between them
(collapse of the curve joining the two triple junctions of $\SS_t$);
\item a $2$--point at an end--point of the network $\SS_t$ where the
  two concurring curves form an angle of $120$
degrees among them (collapse of the curve joining a triple junction to such end--point of $\SS_t$).
\end{itemize}

The same conclusion clearly holds if $\SS_0$ is a tree and $T$ is finite.

If instead the time $T$ coincides with the vanishing time of a loop of the network, by Proposition~\ref{loop}, the curvature is unbounded and there must exists a reachable point for the flow $x_0\in\Omega$ and a sequence of times $t_j\to T$ such that, the associate sequence of rescaled networks $\widetilde{\SS}_{x_0,\tt_{j}}$, as in Proposition~\ref{resclimit}, 
converges in $C^{1,\alpha}\loc\cap W^{2,2}\loc$, for any $\alpha \in (0,1/2)$,
to a limit degenerate regular shrinker $\widetilde\SS_\infty$ which is either a Brakke spoon, or a standard lens or a fish.

If $T=+\infty$, hence $\SS_0$ is a tree, then $\mathbb{S}_t$ converges, as $t\to+\infty$, to a regular network with zero curvature (a stationary point for the length functional). Indeed, as the total length of the network decreases, we have the estimate
\begin{equation}\label{1}
\int_0^{+\infty}\int_{\mathbb{S}_t}k^2\,ds\,dt\leq L(0)<+\infty\,,
\end{equation}
by the first equation in Proposition~\ref{ppp1}. Then, suppose by contradiction that for a sequence of times $t_j\nearrow+\infty$
we have $\int_{\mathbb{S}_{t_j}}k^2\,ds\geq\delta$ for some $\delta>0$. 
By the following estimate, which is inequality~(10.4) in Lemma~10.23 of~\cite{mannovplusch}, 
$$
\frac{d}{dt}\int_{\mathbb{S}_t}k^2\,ds\leq C\Bigl( 1+\Bigl( \int_{\mathbb{S}_t}k^2\Bigr)\Bigr)^3\,,
$$
holding (in the case of fixed end--points) with a uniform constant $C$ independent of time, 
we would have $\int_{\mathbb{S}_{\widetilde{t}}}k^2\,ds\geq\frac{\delta}{2}$, for every $\widetilde{t}$ in a uniform neighborhood of every $t_j$. This is clearly in contradiction with the estimate~\eqref{1}.
Hence, $\lim_{t\to+\infty}\int_{\mathbb{S}_t}k^2\,ds=0$ and, consequently, for every sequence of times $t_i\to+\infty$, there exists a subsequence (not relabeled) such that the evolving networks $\SS_{t_i}$ converge in $C^{1,\alpha}\cap W^{2,2}$, for every $\alpha\in(0,1/2)$,
to a possibly degenerate regular network with zero curvature (hence, ``stationary'' for the length functional), as $i\to\infty$.
\end{proof}

\begin{rem} We underline that, in taking the limit of $\SS_{t_i}$, as $t_i\to T=+\infty$, one or more curves could collapse (possibly to an end--point).
\end{rem}

\begin{prop}\label{vanishing}
Let $\mathbb{S}_0$ be a network with two triple junctions and without end--points on $\partial\Omega$
and $\mathbb{S}_t$ an evolution by curvature in $[0,T)$, with $T<+\infty$.
Then, as $t\to T$, the total length of the network $L(t)$ cannot go to zero.
\end{prop}

\begin{proof}
The network $\mathbb{S}_0$ can only be a $\Theta$--shaped or an eyeglasses--shaped network, as in the following figure, indeed, if some end--points are present, clearly the total length cannot go to zero.
\begin{figure}[H]
\begin{center}
\begin{tikzpicture}[scale=1]
\draw[shift={(0,0)}] 
(-1.73,-1.8) 
to[out= 180,in=180, looseness=1] (-2.8,0) 
to[out= 60,in=150, looseness=1.5] (-1.5,1) 
(-2.8,0)
to[out=-60,in=180, looseness=0.9] (-1.25,-0.75)
(-1.5,1)
to[out= -30,in=90, looseness=0.9] (-1,0)
to[out= -90,in=60, looseness=0.9] (-1.25,-0.75)
to[out= -60,in=0, looseness=0.9](-1.73,-1.8);
\path[shift={(0,0)}] 
(-1.25,-0.85)node[right]{$O^2$}
 (-1.5,-0.35)[left] node{$\gamma^2$}
 (-0.6,.9)[left] node{$\gamma^1$}
 (-0.6,-1.45)[left] node{$\gamma^3$}
 (-3,0.55) node[below] {$O^1$};
 \draw[shift={(5,0)}]
(-2,0) 
to[out= 170,in=40, looseness=1] (-2.9,1.2) 
to[out= -140,in=90, looseness=1] (-3.2,0)
(-2,0)
to[out= -70,in=0, looseness=1] (-3,-0.9) 
to[out= -180,in=-90, looseness=1] (-3.2,0)
(-2,0) 
to[out= 50,in=180, looseness=1] (-1.3,0) 
to[out= 60,in=150, looseness=1.5] (-0.75,1) 
(-1.3,0)
to[out= -60,in=-120, looseness=0.9] (-0.5,-0.75)
(-0.75,1)
to[out= -30,in=90, looseness=0.9] (-0.25,0)
to[out= -90,in=60, looseness=0.9] (-0.5,-0.75);
\path[shift={(5,0)}]
(-1.9,-0.25) node[left]{$O^1$}
(-1.3,0)node[right]{$O^2$}
(-2.9,0.8) node[below] {$\gamma^1$}
(-1.3,0.37)[left] node{$\gamma^3$}
(0,-0.95)[left] node{$\gamma^2$};
\draw[shift={(10,0)}] 
(-3,0) 
to[out= 170,in=140, looseness=1] (-2.1,1.4) 
to[out= -40,in=90, looseness=1] (0,0)
(-3,0) 
to[out= -70,in=-180, looseness=1] (-1,-1.3) 
to[out= 0,in=-90, looseness=1] (0,0)
(-3,0) 
to[out= 50,in=180, looseness=1] (-2.3,0) 
to[out= 60,in=150, looseness=1.5] (-1.75,1) 
(-2.3,0)
to[out= -60,in=-120, looseness=0.9] (-1.5,-0.75)
(-1.75,1)
to[out= -30,in=90, looseness=0.9] (-1.25,0)
to[out= -90,in=60, looseness=0.9] (-1.5,-0.75);
\path[shift={(10,0)}] 
(-2.9,-0.20) node[left]{$O^2$}
(-2.3,0)node[right]{$O^1$}
    (-2.5,0.65) node[below] {$\gamma^3$}
    (-0.7,0.5)[left] node{$\gamma^1$}
    (0,-0.5)[left] node{$\gamma^2$};
\end{tikzpicture}
\end{center}
\begin{caption}{A $\Theta$--shaped network and two different embeddings 
in $\mathbb{R}^2$ of eyeglasses--shaped networks (type A and type B).\label{fishshape}}
\end{caption}
\end{figure}
Consider first the case of a $\Theta$--shaped network.
For both regions the equation of the evolution of the area is 
$$
A'(t)=-\frac{4\pi}{3}\,,
$$
as shown in equation~\eqref{evolarea}.
If $A^1(t)\neq A^2(t)$, then a loop shrinks before the other and $\lim_{t\to T}L(t)\neq 0$.
Hence, $A^1(t)=A^2(t)=4\pi(T-t)/3$, for every $t\in[0,T)$. Taking a blow--up limit $\widetilde{\mathbb{S}}_\infty$ at a hypothetical vanishing point $x_0\in\Omega$, such limit also must contain two loops with equal finite area, since every rescaled network of the sequence $\widetilde{\SS}_{x_0,\tt}$, converging to $\widetilde{\mathbb{S}}_\infty$, 
contains two regions with area equal to $2\pi/3$ (the rescaling factor is $1/\sqrt{2(T-t)}$, see Section~\ref{rescaling}) and the two loops cannot vanish, going to infinity (neither collapsing to a core by the constant area), because they are contiguous and at least one is present in the possible limit shrinker (Brakke spoon, lens or fish). Then, $\widetilde{\mathbb{S}}_\infty$ cannot be a Brakke spoon, a standard lens or a fish, but the curvature must be unbounded, by Proposition~\ref{loop}, hence, this situation is not possible.

We now analyse an eyeglasses--shaped network of ``type B'' (see Figure~\ref{fishshape}).
We call $A^1$ the area enclosed in the curve $\gamma^1$,
$A^2$ the area between $\gamma^1$ and $\gamma^2$ and $A^3$ the sum of $A^1$
and $A^2$. Arguing, as before, by means of equation~\eqref{evolarea} and Gauss--Bonnet theorem, we get that it must be
$$
A^1(t)=5\pi(T-t)/3,\qquad
A^2(t)=2\pi(T-t)/3,\qquad
A^3(t)=7\pi(T-t)/3.
$$
Again, the two loops cannot vanish in the rescaling procedure by the same argument of the previous case and we exclude also this situation by the lack of a shrinker with two regions.

Arguing as before, in the case of a eyeglasses--shaped network of ``type A'' as in Figure~\ref{fishshape}, the evolution equation for the area of the regions is 
$$
A'(t)=-\frac{5\pi}{3}\,,
$$
but, in this situation, we cannot exclude a priori that one of the two loops goes to infinity along the converging sequence of rescaled networks, getting a Brakke spoon as blow--up limit (lens and fish are clearly not possible because of the eyeglasses topology). Anyway, following the proof of Proposition~\ref{cross}, when the blow--up limit around a point $x_0\in\Omega$ is a Brakke spoon, there exists a small annulus around $x_0$ and a time $t_0\in[0,T)$ such that, for every $t\in(t_0,T)$, the network $\SS_t$ in such annulus is a graph over a (piece of a) halfline through the point $x_0$. In particular, $\SS_t$ does not collapse to the point $x_0$, and we have a contradiction. Hence, also this case is impossible.
\end{proof}

\begin{rem}
The previous argument also implies that in a situation of symmetry for the $\Theta$--shaped network (that is
equal area of the two cells) the only possible singularity is a $4$--point formation:
the limit of the length of a curve that connects the two $3$--points goes to zero, as $t\to T$, and the curvature remains bounded.
The same holds in the case of symmetric eyeglasses of ``type A'' and of eyeglasses of ``type B'', if $A^1(0)/A^2(0)=5/2$.
\end{rem}

\section{Singularity formation in explicit cases and restarting the flow}

We state a special case of a theorem of Ilmanen, Neves and Schulze~\cite[Theorem~1.1]{Ilnevsch}, adapted to our situation, regarding the short time existence of a motion by curvature starting from a non--regular network, allowing us to continue the flow after the collision of the two triple junctions.

\begin{teo}\label{restarting}
Let $\mathbb{S}_T$ be a non--regular, connected, embedded, $C^1$ network
with bounded curvature having a single 4--point with the four concurring curves having unit tangent vectors 
forming angles of $120$ and $60$ degrees,
which is $C^2$ away from the $4$--point.
Then, there exists $\widetilde{T}>T$ and a smooth flow of connected regular networks $\mathbb{S}_{t}$, 
locally tree--like for $t\in(T,\widetilde{T})$, such that $\mathbb{S}_t$ is a regular Brakke flow for $t\in[T,\widetilde{T})$.
Moreover, away from the 4--point of $\mathbb{S}_{T}$, the convergence of $\mathbb{S}_t$ to $\mathbb{S}_T$,
as $t\rightarrow T^-$ is in $C^2\loc$ (or as smooth as $\mathbb{S}_0$).\\
Furthermore, there exists a constant $C>0$ such that $\sup_{\mathbb{S}_{t}}\left|k\right|\leq C/\sqrt{t-T}$ and the length of the shortest curve of $\mathbb{S}_{t}$ is bounded from below by $\sqrt{t-T}/C$, for all $t\in (T,\widetilde T)$.
\end{teo}

The following figure shows locally the singularity formation and the restarting of the flow, as described in this theorem, that we call a ``standard'' transition.
\begin{figure}[H]
\begin{center}
\begin{tikzpicture}[rotate=90,scale=0.75]
\draw[color=black!20!white, shift={(0,-3.3)}]
(-0.05,2.65)to[out= -90,in=150, looseness=1] (0.17,2.3)
(0.17,2.3)to[out= -30,in=100, looseness=1] (-0.12,2)
(-0.12,2)to[out= -80,in=40, looseness=1] (0.15,1.7)
(0.15,1.7)to[out= -140,in=90, looseness=1.3](0,1.1)
(0,1.1)--(-.2,1.35)
(0,1.1)--(+.2,1.35);
\draw[color=black!20!white, shift={(0,-9.3)}]
(-0.05,2.65)to[out= -90,in=150, looseness=1] (0.17,2.3)
(0.17,2.3)to[out= -30,in=100, looseness=1] (-0.12,2)
(-0.12,2)to[out= -80,in=40, looseness=1] (0.15,1.7)
(0.15,1.7)to[out= -140,in=90, looseness=1.3](0,1.1)
(0,1.1)--(-.2,1.35)
(0,1.1)--(+.2,1.35);
\draw[color=black]
(-0.05,2.65)to[out= 30,in=180, looseness=1] (2,3)
(-0.05,2.65)to[out= -90,in=150, looseness=1] (0.17,2.3)
(-0.05,2.65)to[out= 150,in=-20, looseness=1] (-2,3.3)
(0.17,2.3)to[out= -30,in=100, looseness=1] (-0.12,2)
(-0.12,2)to[out= -80,in=40, looseness=1] (0.15,1.7)
(0.15,1.7)to[out= -140,in=90, looseness=1](0,1.25)
(0,1.25)to[out= -30,in=180, looseness=1] (1.9,0.7)
(0,1.25)to[out= -150,in=-15, looseness=1] (-1.9,1.2);
\draw[color=black,dashed]
(-2,3.3)to[out= 160,in=-20, looseness=1](-2.7,3.5)
 (-1.9,1.2)to[out= 165,in=-15, looseness=1](-2.7,1.3)
(1.9,0.7)to[out= 0,in=-160, looseness=1] (2.6,0.9)
(2,3)to[out= 0,in=160, looseness=1] (2.8,2.9);
\draw[color=black!30!white,shift={(0,-7)}]
(0,2.65)--(1.73,3.65)
(0,2.65)--(1.73,1.65)
(0,2.65)--(-1.73,3.65)
(0,2.65)--(-1.73,1.65);
\draw[color=black,shift={(0,-7)}]
(0,2.65)to[out= -30,in=180, looseness=1] (1.9,2)
(0,2.65)to[out= -150,in=-15, looseness=1] (-1.9,2.3)
(0,2.65)to[out= 30,in=180, looseness=1] (2.2,3.3)
(0,2.65)to[out= 150,in=-20, looseness=1] (-2.2,3.1);
\draw[color=black,dashed,shift={(0,-7)}]
(-2.2,3.1)to[out= 160,in=-20, looseness=1](-3,3.3)
 (-1.9,2.3)to[out= 165,in=-15, looseness=1](-2.7,2.4)
(1.9,2)to[out= 0,in=-160, looseness=1] (2.6,2.2)
(2.2,3.3)to[out= 0,in=160, looseness=1] (3,3.2);
\path[font=\small]
(-1,-1.5) node[above]{$t\to T$}
(-1,-7.5) node[above]{$t>T$}
(2.8,-.2) node[below]{$\SS_t$}
(2.8,-12.4) node[below]{$\SS_t$}
(2.8,-6) node[below]{$\SS_T$};
\draw[color=black,scale=0.75,shift={(0,-16)},rotate=-30]
(0,2.65)to[out= -30,in=180, looseness=1] (1.9,2.1);
\draw[color=black,scale=0.75,dashed,shift={(0,-16)},rotate=-30]
(1.9,2.1)to[out= 0,in=-160, looseness=1] (2.8,2.4);
\draw[color=black,scale=0.75,shift={(2.65,-16)},rotate=30]
(0,2.65)to[out= 30,in=180, looseness=1] (2,3);
\draw[color=black,dashed,scale=0.75,shift={(2.65,-16)},rotate=30]
(2,3)to[out= 0,in=160, looseness=1] (2.9,2.9);
\draw[color=black,scale=0.75,shift={(0,-16)},rotate=30]
(0,2.65)to[out= -150,in=-15, looseness=1] (-1.9,2.6);
\draw[color=black,dashed,scale=0.75,shift={(0,-16)},rotate=30]
(-1.9,2.6)to[out= 165,in=-15, looseness=1](-2.8,3.0);
\draw[color=black,scale=0.75,shift={(-2.65,-16)},rotate=-30]
(0,2.65)to[out= 150,in=-20, looseness=1] (-2,3.3);
\draw[color=black,dashed,scale=0.75,shift={(-2.65,-16)},rotate=-30]
(-2,3.3)to[out= 160,in=-20, looseness=1](-3,3.5);
\draw[color=black,scale=0.75,shift={(0,-6)}]
(-1.32,-7.7)to[out= 0,in=-150, looseness=1]
(-0.65,-7.4)to[out= 30,in=150, looseness=1]
(0.65,-8)to[out= -30,in=180, looseness=1](1.32,-7.7);
\end{tikzpicture}
\end{center}
\begin{caption}{The local description of a ``standard'' transition.\label{rest}}
\end{caption}
\end{figure}

\begin{rem}\label{nosym}
Notice that the transition, passing by $\SS_T$, is not symmetric: when $\SS_t\to\SS_T$, as $t\to T^-$, the unit tangents, hence the four angles between the curves, are continuous, while when $\SS_t\to\SS_T$, as $t\to T^+$, there is a ``jump'' in such angles, precisely, there is a ``switch'' between the angles of $60$ degrees and the angles of $120$ degrees.
\end{rem}

\begin{rem}
A regular $C^2$  network $\SS=\bigcup_{i=1}^n\sigma^i(I^i)$ 
is called a {\em self--expander} if at every point $x\in\SS$ there holds
\begin{equation*}\label{expandeq}
\underline{k} - x^\perp=0\,. 
\end{equation*}
Let $x_0$ be the 4--point of $\mathbb{S}_T$ and consider the rescalings 
$$
\widetilde{\SS}_{x_0,\tt}=\frac{\mathbb{S}_{t}-x_0}{\sqrt{2(t-T)}}\,,
$$
with $\tt(t)=-\frac12\log(t-T)$.
Then as $\tt\to+\infty$ the rescaled networks $\widetilde{\SS}_{x_0,\tt}$
tend to the unique connected self--expander $\widetilde{\SS}_\infty$ which ``arises'' from the network given by the union of the halflines from the origin 
generated by the unit tangent vectors of the four concurring curves at $x_0$ (see~\cite{mazsae}).
\end{rem}

\begin{rem}
For a general network, a flow of an initial non--regular network given by the general version of the above theorem, is not unique, even if 
$\mathbb{S}_T$ is composed only by halflines from the origin and we search a solution between the connected, tree--like self--expanding networks.
In the particular situation of four halflines forming angles of $60/120$ degrees, instead there exists exactly only one connected, tree--like self--expanding solution (see~\cite[Corollary~11.17]{mannovplusch}).

Considering only the {\em locally connected} network flows has a clear ``physical'' meaning: such a choice ensures that initially separated regions 
remain separated during the flow.
\end{rem}

\begin{rem}
Notice that Theorem~\ref{restarting} gives only a short time existence result,
indeed, it is not possible to say in general if and when another singularity could appear. In particular, we are not able to exclude that the singular times may accumulate.
\end{rem}

\begin{rem}
In the actual formulation, the ``restarting theorem'' (Theorem~1.1 in~\cite{Ilnevsch} or Theorem~11.1 in~\cite{mannovplusch}) works after the onset of the first singularity
only in the case in which the curvature remains bounded and a triple junction does not
collapse to an end--point on the boundary of $\Omega$ (case $1$ of Theorem~\ref{main}).
However, with a (non trivial) modification (see Remark~11.20 in~\cite{mannovplusch}) of such theorem, we should be able to restart the flow also after the appearance of general singularities in which some regions collapse and the curvature is not bounded (case $3$ of Theorem~\ref{main}).\\
We will assume the validity of such extension in the following description of the behavior of the networks.

For rigorous proofs about the convergence, as $t\to T$, of the network $\SS_t$ to a limit network $\SS_T$ and the restarting of the flow,
we refer the reader to Sections~11,~12 and~13 in~\cite{mannovplusch}.
\end{rem}

\subsection{The tree}

This is the only network with two triple junctions which does not present loops. Consequently, it is the only case
 where we could  have global existence.
\begin{figure}[H]
\begin{center}
\begin{tikzpicture}[scale=1]
\draw[black!70!black]
 (-3.73,0)
to[out= 50,in=180, looseness=1] (-2.3,0.7)
to[out= 60,in=180, looseness=1.5] (-0.45,1.55)
(-2.3,0.7)
to[out= -60,in=130, looseness=0.9] (-1,-0.3)
to[out= 10,in=100, looseness=0.9](0.1,-0.8)
(-1,-0.3)
to[out=-110,in=50, looseness=0.9](-2.7,-1.7);
\draw[color=black!50!white,scale=1,domain=-3.141: 3.141,
smooth,variable=\t,shift={(-1.72,0)},rotate=0]plot({2.*sin(\t r)},
{2.*cos(\t r)}) ;
\path
 (-3.73,0) node[left]{$P^1$}
 (-2.7,-1.77)node[below]{$P^2$}
 (0.1,-0.8)node[right]{$P^3$}
  (-0.43,1.6) node[right]{$P^4$}
   (-3,0.6) node[below] {$\gamma^1$}
   (-1.5,1.3) node[right] {$\gamma^4$}
   (-1.1,-1.2)[left] node{$\gamma^2$}
   (0,-0.8)[left] node{$\gamma^3$}
    (-1.3,0.5)[left] node{$\gamma^5$}
  (-2.4,1.3) node[below] {$O^1$}
   (-0.8,0.3) node[below] {$O^2$};
\end{tikzpicture}
\end{center}
\begin{caption}{The tree.\label{tre}}
\end{caption}
\end{figure}
The behavior of the flow $\SS_t$ is described by Theorem~\ref{main}. We only mention that if at the maximal time $T$ no ``boundary'' curve collapses, then $\SS_t$ converges to a limit network $\SS_T$ with bounded curvature and, restarting the flow by means of Theorem~\ref{restarting}, we get another regular tree which is the only ``other'' possible connected, regular tree, joining the four fixed end--points  $P^1$, $P^2$, $P^3$ and $P^4$. That is, the ``standard'' transition at time $T$, as in Figure~\ref{rest}, transforms one in the other and viceversa. The natural question, if during the flow there could appear infinite singular times (and also whether they could ``accumulate'') producing an ``oscillation'' phenomenon between the two structures, has no answer at the moment.

\begin{figure}[H]
\begin{center}
\begin{tikzpicture}[scale=0.6]
\draw[color=black!50!white,rotate=90,shift={(0,-1.9)}, scale=0.6]
(-0.05,2.65)to[out= -90,in=150, looseness=1] (0.17,2.3)
(0.17,2.3)to[out= -30,in=100, looseness=1] (-0.12,2)
(-0.12,2)to[out= -80,in=40, looseness=1] (0.15,1.7)
(0.15,1.7)to[out= -140,in=90, looseness=1.3](0,1.1)
(0,1.1)--(-.2,1.35)
(0,1.1)--(+.2,1.35);
\draw[color=black!50!white,rotate=90,shift={(0,-6.9)}, scale=0.6]
(-0.05,2.65)to[out= -90,in=150, looseness=1] (0.17,2.3)
(0.17,2.3)to[out= -30,in=100, looseness=1] (-0.12,2)
(-0.12,2)to[out= -80,in=40, looseness=1] (0.15,1.7)
(0.15,1.7)to[out= -140,in=90, looseness=1.3](0,1.1)
(0,1.1)--(-.2,1.35)
(0,1.1)--(+.2,1.35);
\draw[color=black!50!white,rotate=90,shift={(0,-11.9)}, scale=0.6]
(-0.05,2.65)to[out= -90,in=150, looseness=1] (0.17,2.3)
(0.17,2.3)to[out= -30,in=100, looseness=1] (-0.12,2)
(-0.12,2)to[out= -80,in=40, looseness=1] (0.15,1.7)
(0.15,1.7)to[out= -140,in=90, looseness=1.3](0,1.1)
(0,1.1)--(-.2,1.35)
(0,1.1)--(+.2,1.35);
\draw[color=black!50!white,rotate=90,shift={(0,-16.9)}, scale=0.6]
(-0.05,2.65)to[out= -90,in=150, looseness=1] (0.17,2.3)
(0.17,2.3)to[out= -30,in=100, looseness=1] (-0.12,2)
(-0.12,2)to[out= -80,in=40, looseness=1] (0.15,1.7)
(0.15,1.7)to[out= -140,in=90, looseness=1.3](0,1.1)
(0,1.1)--(-.2,1.35)
(0,1.1)--(+.2,1.35);
\draw[black]
 (-3.03,1.25) 
to[out= -50,in=180, looseness=1] (-2,0.6) 
to[out= 60,in=180, looseness=1.5] (-0.43,1.25) 
(-2,0.6)
to[out= -60,in=130, looseness=0.9] (-1.5,-0.3)
to[out= 10,in=100, looseness=0.9](-0.43,-1.25)
(-1.5,-0.3)
to[out=-110,in=50, looseness=0.9](-3.03,-1.25);
\draw[black, shift={(5,0)}]
 (-3.03,1.25) 
to[out= -50,in=180, looseness=1]  (-1.65,0.25)
to[out= 60,in=180, looseness=1.5] (-0.43,1.25) 
(-1.65,0.25)
to[out= 0,in=100, looseness=0.9](-0.43,-1.25)
(-1.65,0.25)
to[out=-120,in=50, looseness=0.9](-3.03,-1.25);
\draw[black, shift={(10,0)}]
 (-3.03,1.25) 
to[out= -50,in=180, looseness=1]
(-2.4,-0.2)
to[out= -60,in=50, looseness=1.5] (-3.03,-1.25)
(-2.4,-0.2)
to[out= 60,in=-130, looseness=0.9]  (-1.2,0.3)
to[out= 110,in=100, looseness=0.9](-0.43,1.25) 
(-1.2,0.3)
to[out=-10,in=50, looseness=0.9](-0.43,-1.25);

\draw[black, shift={(20,0)}]
  (-3.03,1.25) 
to[out= -40,in=150, looseness=1] (-1.8,0.5) 
to[out= 30,in=180, looseness=1.5] (-0.43,1.25) 
(-1.8,0.5)
to[out= -90,in=90, looseness=0.9] (-1.8,-0.5)
to[out= -30,in=110, looseness=0.9](-0.43,-1.25)
(-1.8,-0.5)
to[out=-150,in=10, looseness=0.9](-3.03,-1.25);

\draw[black, shift={(15,0)}]
 (-3.03,1.25) 
to[out= -50,in=120, looseness=1]  (-1.75,0.15)
to[out= 60,in=180, looseness=1.5] (-0.43,1.25) 
(-1.75,0.15)
to[out= -60,in=100, looseness=0.9](-0.43,-1.25) 
(-1.75,0.15)
to[out=-120,in=50, looseness=0.9] (-3.03,-1.25) ;
\draw[color=black, domain=-3.141: 3.141,
smooth,variable=\t,shift={(-1.72,0)},rotate=0, scale=0.9]plot({2.*sin(\t r)},
{2.*cos(\t r)}) ;
\draw[color=black, domain=-3.141: 3.141,
smooth,variable=\t,shift={(3.28,0)},rotate=0, scale=0.9]plot({2.*sin(\t r)},
{2.*cos(\t r)}) ;
\draw[color=black, domain=-3.141: 3.141,
smooth,variable=\t,shift={(8.28,0)},rotate=0, scale=0.9]plot({2.*sin(\t r)},
{2.*cos(\t r)}) ;
\draw[color=black, domain=-3.141: 3.141,
smooth,variable=\t,shift={(13.28,0)},rotate=0, scale=0.9]plot({2.*sin(\t r)},
{2.*cos(\t r)}) ;
\draw[color=black, domain=-3.141: 3.141,
smooth,variable=\t,shift={(18.28,0)},rotate=0, scale=0.9]plot({2.*sin(\t r)},
{2.*cos(\t r)}) ;
\end{tikzpicture}
\end{center}
\begin{caption}{The ``standard'' transition for a tree--shaped network.}
\end{caption}
\end{figure}

\subsection{The lens and the island}

We start considering a lens--shaped network.
\begin{figure}[H]
\begin{center}
\begin{tikzpicture}[scale=1]
\draw[color=black!70!black,shift={(5,0)}]
(-4.7,1)
to[out= -50,in=180, looseness=1] (-2.8,0)
to[out= 60,in=150, looseness=1.5] (-1.5,1)
(-2.8,0)
to[out=-60,in=180, looseness=0.9] (-1.25,-0.75)
(-1.5,1)
to[out= -30,in=90, looseness=0.9] (-1,0)
to[out= -90,in=60, looseness=0.9] (-1.25,-0.75)
to[out= -60,in=150, looseness=0.9](1,-1.3);
\draw[color=black!50!white,scale=1,domain=-3.15: 3.15,
smooth,variable=\t,shift={(3.28,0)},rotate=0]plot({3.25*sin(\t r)},
{2.5*cos(\t r)}) ;
\path[shift={(5,0)}]
(-5.6,-0.75)node[right]{$\Omega$}
(-4.7,1) node[left]{$P^1$}
(1,-1.3)node[right]{$P^2$}
(-3,0) node[below] {$O^1$}
(-2,0.5) node[below] {$A$}
(-1.2,-0.7)node[right]{$O^2$}
(-1.5,-1)[left] node{$\gamma^2$}
(-3.6,0.89) node[below] {$\gamma^1$}
(-0.6,0.9)[left] node{$\gamma^4$}
(0.7,-0.75)[left] node{$\gamma^3$};
\end{tikzpicture}
\end{center}
\begin{caption}{The lens.\label{lens}}
\end{caption}
\end{figure}
By Theorem~\ref{main} the maximal time $T$ is finite. 
Suppose that the length of both curves $\gamma^2$ and $\gamma^4(t)$ goes to zero, as $t\to T$, then, 
letting $A(0)$ be the area of the bounded region at time $t=0$, we have $T=\frac{3A(0)}{4\pi}$, such region is collapsing,  the curvature is unbounded and the blow--up limit around the collapsing point is a lens--shaped or fish--shaped regular shrinker. 
As $t\to T$, the network $\SS_t$ converges to a single curve joining $P^1$ and $P^2$ which is $C^1$ in the first case  and with an angle in the second one, then, restarting the flow, it becomes immediately smooth. Notice that the bounded region cannot collapse to an end--point, by Lemma~\ref{trinoncoll}.

If $T<\frac{3A(0)}{4\pi}$, the bounded region cannot collapse (by equation~\eqref{evolarea}) and the network is locally uniformly a tree, hence, the curvature stays bounded, a curve collapses and $\SS_t$ converges to a limit network $\SS_T$. Notice that only one curve can collapse, by Lemma~\ref{trinoncoll} and the fact that the region does not collapse. If such collapsing curve is not a ``boundary'' one, restarting the flow by means of Theorem~\ref{restarting}, we get an island--shaped regular network.
\begin{figure}[H]
\begin{center}
\begin{tikzpicture}[scale=1]
\draw[color=black!70!black, shift={(10,0)}]
(-2,0)
to[out= 170,in=40, looseness=1] (-2.9,1.2)
to[out= -140,in=90, looseness=1] (-3.2,0)
(-2,0)
to[out= -70,in=0, looseness=1] (-3,-0.9)
to[out= -180,in=-90, looseness=1] (-3.2,0)
(-2,0)
to[out= 50,in=180, looseness=1] (-1.3,0)
to[out= 60,in=150, looseness=1.5] (-0.75,1)
(-1.3,0)
to[out= -60,in=-120, looseness=0.9] (-0.5,-0.75)
(-0.75,1)
to[out= -30,in=-150, looseness=0.9] (1.75,1)
(0,-2.32)
to[out= 90,in=60, looseness=0.9] (-0.5,-0.75);
\draw[color=black!50!white,scale=1,domain=-3.141: 3.141,
smooth,variable=\t,shift={(8.78,0)},rotate=0]plot({3.25*sin(\t r)},
{2.5*cos(\t r)}) ;
\path[shift={(10,0)}]
(-2,-0.2) node[left]{$O^1$}
(-1.3,0)node[right]{$O^2$}
(-2.85,0.8) node[below] {$\gamma^1$}
(-1.2,0.3)[left] node{$\gamma^2$}
(-.1,-1.3)[left] node{$\gamma^3$}
(1,1.5) node[below] {$\gamma^4$}
(0.5,-2.7)[left] node{$P^2$}
(2.5,1.35)[left] node{$P^1$};
\end{tikzpicture}
\end{center}
\begin{caption}{The island.\label{island}}
\end{caption}
\end{figure}
In this case, as before, the maximal time of existence $T$ of a smooth flow is finite and bounded by $\frac{3A(0)}{5\pi}$, where $A(0)$ is the area of the bounded region at time zero. 

If the length of the curve $\gamma^1$ goes to zero as $t\to T$, then
$T=\frac{3A(0)}{5\pi}$ and the curvature is unbounded. Notice that the curve joining the bounded region with the rest of the network cannot collapse at the same time, by repeating the argument in the final part of the proof of Proposition~\ref{vanishing}. Hence, the only possible blow--up limit around the collapsing point is a Brakke spoon. Moreover, Lemma~\ref{trinoncoll} and maximum principle also do not allow the region to collapse to an end--point of the network. Anyway, it could happen that a ``boundary'' curve collapses at the same time. The network $\SS_t$, as $t\to T$, converges to a network with a $C^1$ curve ending at the point where the region collapses. 

If $T<\frac{3A(0)}{5\pi}$, the bounded region cannot collapse (by equation~\eqref{evolarea}) and the network is locally uniformly a tree, hence, the curvature stays bounded, a curve collapses and $\SS_t$ converges to a limit network $\SS_T$. As above only one curve can collapse. If such collapsing curve is not a ``boundary'' one, restarting the flow by means of Theorem~\ref{restarting}, we get a lens--shaped regular network.

\begin{figure}[H]
\begin{center}
\begin{tikzpicture}[scale=0.65]
\draw[color=black!50!white,rotate=90,shift={(0,-1.9)}, scale=0.6]
(-0.05,2.65)to[out= -90,in=150, looseness=1] (0.17,2.3)
(0.17,2.3)to[out= -30,in=100, looseness=1] (-0.12,2)
(-0.12,2)to[out= -80,in=40, looseness=1] (0.15,1.7)
(0.15,1.7)to[out= -140,in=90, looseness=1.3](0,1.1)
(0,1.1)--(-.2,1.35)
(0,1.1)--(+.2,1.35);
\draw[color=black!50!white,rotate=90,shift={(0,-6.9)}, scale=0.6]
(-0.05,2.65)to[out= -90,in=150, looseness=1] (0.17,2.3)
(0.17,2.3)to[out= -30,in=100, looseness=1] (-0.12,2)
(-0.12,2)to[out= -80,in=40, looseness=1] (0.15,1.7)
(0.15,1.7)to[out= -140,in=90, looseness=1.3](0,1.1)
(0,1.1)--(-.2,1.35)
(0,1.1)--(+.2,1.35);
\draw[color=black!50!white,rotate=90,shift={(0,-11.9)}, scale=0.6]
(-0.05,2.65)to[out= -90,in=150, looseness=1] (0.17,2.3)
(0.17,2.3)to[out= -30,in=100, looseness=1] (-0.12,2)
(-0.12,2)to[out= -80,in=40, looseness=1] (0.15,1.7)
(0.15,1.7)to[out= -140,in=90, looseness=1.3](0,1.1)
(0,1.1)--(-.2,1.35)
(0,1.1)--(+.2,1.35);
\draw[color=black!50!white,rotate=90,shift={(0,-16.9)}, scale=0.6]
(-0.05,2.65)to[out= -90,in=150, looseness=1] (0.17,2.3)
(0.17,2.3)to[out= -30,in=100, looseness=1] (-0.12,2)
(-0.12,2)to[out= -80,in=40, looseness=1] (0.15,1.7)
(0.15,1.7)to[out= -140,in=90, looseness=1.3](0,1.1)
(0,1.1)--(-.2,1.35)
(0,1.1)--(+.2,1.35);
\draw[color=black, domain=-3.141: 3.141,
smooth,variable=\t,shift={(-1.72,0)},rotate=0, scale=0.9]plot({2.*sin(\t r)},
{2.*cos(\t r)}) ;
\draw[color=black, domain=-3.141: 3.141,
smooth,variable=\t,shift={(3.28,0)},rotate=0, scale=0.9]plot({2.*sin(\t r)},
{2.*cos(\t r)}) ;
\draw[color=black, domain=-3.141: 3.141,
smooth,variable=\t,shift={(8.28,0)},rotate=0, scale=0.9]plot({2.*sin(\t r)},
{2.*cos(\t r)}) ;
\draw[color=black,domain=-3.141: 3.141,
smooth,variable=\t,shift={(13.28,0)},rotate=0, scale=0.9]plot({2.*sin(\t r)},
{2.*cos(\t r)}) ;
\draw[color=black, domain=-3.141: 3.141,
smooth,variable=\t,shift={(18.28,0)},rotate=0, scale=0.9]plot({2.*sin(\t r)},
{2.*cos(\t r)}) ;
\draw[color=black]  
(-1.6,0.5) 
to[out= 60,in=-150, looseness=1](-0.43,1.25)
(-1.6,0.5) 
to[out= -60,in=150, looseness=1.5] (-1.5,0) 
(-1.6,0.5)
to[out=180,in=90,  looseness=1] (-3.25,-0.625)
to[out=-90,in=180, looseness=0.9] (-1.25,-0.75)
(-1.5,0)
to[out= -30,in=90, looseness=0.9] (-1,0)
to[out= -90,in=60, looseness=0.9] (-1.25,-0.75)
to[out= -60,in=150, looseness=0.9](-0.43,-1.25);
\draw[color=black,shift={(5,0)}]  
(-1.5,0)
to[out= 150,in=40, looseness=1] (-2.9,0.9) 
to[out= -140,in=90, looseness=1] (-3.2,0)
(-1.5,0) 
to[out= -150,in=0, looseness=1] (-3,-0.7) 
to[out= -180,in=-90, looseness=1] (-3.2,0)
(-1.5,0) 
to[out= 30,in=150, looseness=1.5] (-0.43,1.25) 
(-1.5,0)
to[out= -30,in=-120, looseness=0.9](-0.43,-1.25);
\draw[color=black,shift={(10,0)}]  
(-2,0)
to[out= 170,in=40, looseness=1] (-2.5,0.7) 
to[out= -140,in=90, looseness=1] (-2.8,0)
(-2,0) 
to[out= -70,in=0, looseness=1] (-2.65,-0.5) 
to[out= -180,in=-90, looseness=1] (-2.8,0)
(-2,0) 
to[out= 50,in=180, looseness=1] (-1.3,0) 
to[out= 60,in=150, looseness=1.5] (-0.75,1) 
(-1.3,0)
to[out= -60,in=-120, looseness=0.9] (-0.5,-0.75)
(-0.75,1)
to[out= -30,in=90, looseness=0.9](-0.43,1.25)
(-0.43,-1.25)
to[out= -90,in=60, looseness=0.9] (-0.5,-0.75);
\draw[color=black,shift={(15,0)}]  
(-1,0)
to[out= 120,in=40, looseness=1] (-1.9,0.6) 
to[out= -140,in=90, looseness=1] (-2.2,0)
(-1,0) 
to[out= -120,in=0, looseness=1] (-2,-0.4) 
to[out= -180,in=-90, looseness=1] (-2.2,0)
(-1,0) 
to[out= 60,in=-150, looseness=1.5] (-0.43,1.25) 
(-1,0)
to[out= -60,in=-120, looseness=0.9](-0.43,-1.25);
\draw[black, shift={(20,0)}]
(-1,0.65) 
to[out= 45,in=180, looseness=1.5] (-0.43,1.25) 
(-1,0.65)
to[out= -75,in=75, looseness=0.9] (-1,-0.65)
(-1,0.65)
to[out= -195,in=90, looseness=0.9] (-1.5,0)
to[out= -90,in=195, looseness=0.9] (-1,-0.65)
to[out= -45,in=110, looseness=0.9](-0.43,-1.25);
\end{tikzpicture}
\end{center}
\begin{caption}{A ``standard'' transition trough a $4$--point transforms a lens in a island and vice versa.}
\end{caption}
\end{figure}

This discussion shows that the lens and the island shapes are in a sense ``dual'', with the meaning that a ``standard'' transition, as in Figure~\ref{rest}, transforms one in the other and viceversa. As before, we do not know if during the flow this kind of ``oscillation'' phenomenon can happen infinite (possibly accumulating) times.

\subsection{The theta and the eyeglasses}

For a regular theta--shaped network,  as in Figure~\ref{theta}, we call $A^1$ the area enclosed by the curves $\gamma^1$ and $\gamma^2$ and $A^2$ the area enclosed by $\gamma^2$ and $\gamma^3$.
\begin{figure}[H]
\begin{center}
\begin{tikzpicture}[scale=1]
\draw[color=black!70!black,shift={(5,0)}]
(-1.73,-1.8)
to[out= 180,in=180, looseness=1] (-2.8,0)
to[out= 60,in=150, looseness=1.5] (-1.5,1)
(-2.8,0)
to[out=-60,in=180, looseness=0.9] (-1.25,-0.75)
(-1.5,1)
to[out= -30,in=90, looseness=0.9] (-1,0)
to[out= -90,in=60, looseness=0.9] (-1.25,-0.75)
to[out= -60,in=0, looseness=0.9](-1.73,-1.8);
\draw[color=black!50!white,scale=1,domain=-3.15: 3.15,
smooth,variable=\t,shift={(3.3,0)},rotate=0]plot({3.25*sin(\t r)},
{2.5*cos(\t r)}) ;
\path[shift={(5,0)}]
(-5.6,-0.75)node[right]{$\Omega$}
(-1.25,-0.75)node[right]{$O^2$}
 (-1.5,-0.3)[left] node{$\gamma^2$}
  (-1.5,0.5)[left] node{$A^1$}
   (-1.5,-1.1)[left] node{$A^2$}
 (-0.8,1.1)[left] node{$\gamma^1$}
 (-0.6,-1.45)[left] node{$\gamma^3$}
 (-3,0.55) node[below] {$O^1$};
 \end{tikzpicture}
\end{center}
\begin{caption}{The theta.\label{theta}}
\end{caption}
\end{figure}
The analysis at the finite maximal time $T\leq \frac{3}{4\pi}\min\{A^1(0) , A^2(0)\}$ is similar to the previous cases. We know that the two regions cannot shrink at the same time (otherwise the whole network vanishes, which is excluded by Proposition~\ref{vanishing}), if $T=\frac{3}{4\pi}\min\{A^1(0) , A^2(0)\}$, the curvature cannot be bounded and we get a blow--up limit which is a lens or a fish. The network $\SS_t$ converges to a single closed curve, possibly with an angle,
which becomes immediately smooth when we restart the flow.

If $T<\frac{3}{4\pi}\min\{A^1(0) , A^2(0)\}$, then the two regions do not collapse (by equation~\ref{evolarea}) and the network is locally uniformly a tree, hence, the curvature stays bounded, only one curve collapses and $\SS_t$ converges to a limit network $\SS_T$ with bounded curvature. There are then two different cases: if the ``inner'' curve collapses, restarting the flow by means of Theorem~\ref{restarting}, we get a ``type A'' eyeglasses--shaped regular network, if instead one of the two ``external'' curves collapses, restarting the flow, we get a ``type B'' eyeglasses--shaped regular network.
\begin{figure}[H]
\begin{center}
\begin{tikzpicture}[scale=1]
\draw[black!70!black]
(-2,0)
to[out= 170,in=40, looseness=1] (-2.9,1.2)
to[out= -140,in=90, looseness=1] (-3.2,0)
(-2,0)
to[out= -70,in=0, looseness=1] (-3,-0.9)
to[out= -180,in=-90, looseness=1] (-3.2,0)
(-2,0)
to[out= 50,in=180, looseness=1] (-1.3,0)
to[out= 60,in=150, looseness=1.5] (-0.75,1)
(-1.3,0)
to[out= -60,in=-120, looseness=0.9] (-0.5,-0.75)
(-0.75,1)
to[out= -30,in=90, looseness=0.9] (-0.25,0)
to[out= -90,in=60, looseness=0.9] (-0.5,-0.75);
\draw[color=black!50!white,scale=1,domain=-4.141: 4.141,
smooth,variable=\t,shift={(-1.72,0)},rotate=0]plot({2.3*sin(\t r)},
{2.3*cos(\t r)}) ;
\path
(-3,-2.5) node[left]{$\Omega$}
(-1.88,-0.2) node[left]{$O^1$}
(-1.3,0)node[right]{$O^2$}
    (-2.8,0.8) node[below] {$\gamma^1$}
    (-1.3,0.4)[left] node{$\gamma^3$}
    (0,-0.95)[left] node{$\gamma^2$};
\draw[color=black!70!black,shift={(7,0)}]
(-3,0)
to[out= 170,in=140, looseness=1] (-2.1,1.4)
to[out= -40,in=90, looseness=1] (0,0)
(-3,0)
to[out= -70,in=-180, looseness=1] (-1,-1.3)
to[out= 0,in=-90, looseness=1] (0,0)
(-3,0)
to[out= 50,in=180, looseness=1] (-2.3,0)
to[out= 60,in=150, looseness=1.5] (-1.75,1)
(-2.3,0)
to[out= -60,in=-120, looseness=0.9] (-1.5,-0.75)
(-1.75,1)
to[out= -30,in=90, looseness=0.9] (-1.25,0)
to[out= -90,in=60, looseness=0.9] (-1.5,-0.75);
\draw[color=black!50!white,scale=1,domain=-4.141: 4.141,
smooth,variable=\t,shift={(5.28,0)},rotate=0]plot({2.3*sin(\t r)},
{2.3*cos(\t r)}) ;
\path[shift={(7,0)}]
(-3,-2.5) node[left]{$\Omega$}
(-2.85,-0.25) node[left]{$O^2$}
(-2.3,0)node[right]{$O^1$}
    (-2.5,0.72) node[below] {$\gamma^3$}
    (-0.7,0.5)[left] node{$\gamma^1$}
    (0,-0.5)[left] node{$\gamma^2$};
\end{tikzpicture}
\end{center}
\begin{caption}{The eyeglasses: ``type A'' and ``type B''.\label{eyeglasses}}
\end{caption}
\end{figure}
For these networks, if a region (or both in the ``type A'' case) collapses at the maximal time $T$, which is finite, the only possible blow--up limit is a Brakke spoon and 
$\SS_t$ converges to a network $\SS_T$ having, as in the case of the island, a curve with a ``free'' end--point. Otherwise, if there is no collapse of a region, restarting the flow by means of Theorem~\ref{restarting}, we get back to theta--shaped regular network.

As before, there is a sort of ``duality'' between the theta and the eyeglasses shapes: a ``standard'' transition transforms one in the other and viceversa. Again, we do not know if this kind of ``oscillations'' can happen infinite times.

\begin{figure}[H]
\begin{center}
\begin{tikzpicture}[scale=0.65]
\draw[color=black!50!white,rotate=90,shift={(0,-1.9)}, scale=0.6]
(-0.05,2.65)to[out= -90,in=150, looseness=1] (0.17,2.3)
(0.17,2.3)to[out= -30,in=100, looseness=1] (-0.12,2)
(-0.12,2)to[out= -80,in=40, looseness=1] (0.15,1.7)
(0.15,1.7)to[out= -140,in=90, looseness=1.3](0,1.1)
(0,1.1)--(-.2,1.35)
(0,1.1)--(+.2,1.35);
\draw[color=black!50!white,rotate=90,shift={(0,-6.9)}, scale=0.6]
(-0.05,2.65)to[out= -90,in=150, looseness=1] (0.17,2.3)
(0.17,2.3)to[out= -30,in=100, looseness=1] (-0.12,2)
(-0.12,2)to[out= -80,in=40, looseness=1] (0.15,1.7)
(0.15,1.7)to[out= -140,in=90, looseness=1.3](0,1.1)
(0,1.1)--(-.2,1.35)
(0,1.1)--(+.2,1.35);
\draw[color=black!50!white,rotate=90,shift={(0,-11.9)}, scale=0.6]
(-0.05,2.65)to[out= -90,in=150, looseness=1] (0.17,2.3)
(0.17,2.3)to[out= -30,in=100, looseness=1] (-0.12,2)
(-0.12,2)to[out= -80,in=40, looseness=1] (0.15,1.7)
(0.15,1.7)to[out= -140,in=90, looseness=1.3](0,1.1)
(0,1.1)--(-.2,1.35)
(0,1.1)--(+.2,1.35);
\draw[color=black!50!white,rotate=90,shift={(0,-16.9)}, scale=0.6]
(-0.05,2.65)to[out= -90,in=150, looseness=1] (0.17,2.3)
(0.17,2.3)to[out= -30,in=100, looseness=1] (-0.12,2)
(-0.12,2)to[out= -80,in=40, looseness=1] (0.15,1.7)
(0.15,1.7)to[out= -140,in=90, looseness=1.3](0,1.1)
(0,1.1)--(-.2,1.35)
(0,1.1)--(+.2,1.35);
\draw[color=black, domain=-3.141: 3.141,
smooth,variable=\t,shift={(-1.72,0)},rotate=0, scale=0.9]plot({2.*sin(\t r)},
{2.*cos(\t r)}) ;
\draw[color=black, domain=-3.141: 3.141,
smooth,variable=\t,shift={(3.28,0)},rotate=0, scale=0.9]plot({2.*sin(\t r)},
{2.*cos(\t r)}) ;
\draw[color=black, domain=-3.141: 3.141,
smooth,variable=\t,shift={(8.28,0)},rotate=0, scale=0.9]plot({2.*sin(\t r)},
{2.*cos(\t r)}) ;
\draw[color=black, domain=-3.141: 3.141,
smooth,variable=\t,shift={(13.28,0)},rotate=0, scale=0.9]plot({2.*sin(\t r)},
{2.*cos(\t r)}) ;
\draw[color=black, domain=-3.141: 3.141,
smooth,variable=\t,shift={(18.28,0)},rotate=0, scale=0.9]plot({2.*sin(\t r)},
{2.*cos(\t r)}) ;
\draw[color=black,shift={(0,0.3)}] 
(-1.73,-1.8) 
to[out= 180,in=180, looseness=1] (-2.8,0) 
to[out= 60,in=150, looseness=1.5] (-1.5,1) 
(-2.8,0)
to[out=-60,in=180, looseness=0.9] (-1.25,-0.75)
(-1.5,1)
to[out= -30,in=90, looseness=0.9] (-1,0)
to[out= -90,in=60, looseness=0.9] (-1.25,-0.75)
to[out= -60,in=0, looseness=0.9](-1.73,-1.8);
\draw[scale=0.33,rotate=90,color=black,shift={(0,-9.7)}]
(0,0)to[out= 30,in=180, looseness=1] (2.7,1.5)
(2.7,1.5)to[out= 0,in=90, looseness=1] (4.73,0)
(0,0)to[out= -30,in=180, looseness=1] (2.7,-1.5)
(2.7,-1.5)to[out= 0,in=-90, looseness=1] (4.73,0);
\draw[scale=0.33,rotate=-90,color=black,shift={(0,9.7)}]
(0,0)to[out= 30,in=180, looseness=1] (2.7,1.5)
(2.7,1.5)to[out= 0,in=90, looseness=1] (4.73,0)
(0,0)to[out= -30,in=180, looseness=1] (2.7,-1.5)
(2.7,-1.5)to[out= 0,in=-90, looseness=1] (4.73,0);
\draw[black, shift={(8.5,2)}, rotate=90]
(-2.25,-0.12) 
to[out= 20,in=180, looseness=1] (-1.55,0.17);
\draw[scale=0.43,color=black,shift={(19.34,-1)}, rotate=90]
(2,0)to[out= 60,in=180, looseness=1] (3.7,1)
(3.7,1)to[out= 0,in=90, looseness=1] (5,0)
(2,0)to[out= -60,in=180, looseness=1] (3.7,-1)
(3.7,-1)to[out= 0,in=-90, looseness=1] (5,0);
\draw[scale=0.43,color=black,shift={(19.34,1.3)}, rotate=-70]
(2,0)to[out= 60,in=180, looseness=1] (3.7,1)
(3.7,1)to[out= 0,in=90, looseness=1] (5,0)
(2,0)to[out= -60,in=180, looseness=1] (3.7,-1)
(3.7,-1)to[out= 0,in=-90, looseness=1] (5,0);
\draw[scale=0.35,rotate=90,color=black,shift={(0,-38)}]
(0,0)to[out= 60,in=180, looseness=1] (2,1.5)
(2,1.5)to[out= 0,in=90, looseness=1] (3.7,0)
(0,0)to[out= -60,in=180, looseness=1] (2,-1.5)
(2,-1.5)to[out= 0,in=-90, looseness=1] (3.7,0);
\draw[scale=0.35,rotate=-90,color=black,shift={(0,38)}]
(0,0)to[out= 60,in=180, looseness=1] (2,1.5)
(2,1.5)to[out= 0,in=90, looseness=1] (3.7,0)
(0,0)to[out= -60,in=180, looseness=1] (2,-1.5)
(2,-1.5)to[out= 0,in=-90, looseness=1] (3.7,0);
\draw[color=black,shift={(20,0)}] 
(-1.73,-1) 
to[out= 180,in=-120, looseness=1](-2.2,0)
to[out= 120,in=150, looseness=1.5] (-1.5,1) 
(-2.2,0)
to[out=0,in=180, looseness=0.9] (-1.4,0)
(-1.5,1)
to[out= -30,in=90, looseness=0.9] (-1.3,0.5)
to[out= -90,in=60, looseness=0.9] (-1.4,0)
to[out= -60,in=0, looseness=0.9](-1.73,-1);
\end{tikzpicture}
\end{center}
\begin{caption}{An example of evolution from a theta-shaped network to an Eyeglasses 
``type A'' passing by a $4$--point formation.}
\end{caption}
\end{figure}

\begin{figure}[H]
\begin{center}
\begin{tikzpicture}[scale=0.65]
\draw[color=black!50!white,rotate=90,shift={(0,-1.9)}, scale=0.6]
(-0.05,2.65)to[out= -90,in=150, looseness=1] (0.17,2.3)
(0.17,2.3)to[out= -30,in=100, looseness=1] (-0.12,2)
(-0.12,2)to[out= -80,in=40, looseness=1] (0.15,1.7)
(0.15,1.7)to[out= -140,in=90, looseness=1.3](0,1.1)
(0,1.1)--(-.2,1.35)
(0,1.1)--(+.2,1.35);
\draw[color=black!50!white,rotate=90,shift={(0,-6.9)}, scale=0.6]
(-0.05,2.65)to[out= -90,in=150, looseness=1] (0.17,2.3)
(0.17,2.3)to[out= -30,in=100, looseness=1] (-0.12,2)
(-0.12,2)to[out= -80,in=40, looseness=1] (0.15,1.7)
(0.15,1.7)to[out= -140,in=90, looseness=1.3](0,1.1)
(0,1.1)--(-.2,1.35)
(0,1.1)--(+.2,1.35);
\draw[color=black!50!white,rotate=90,shift={(0,-11.9)}, scale=0.6]
(-0.05,2.65)to[out= -90,in=150, looseness=1] (0.17,2.3)
(0.17,2.3)to[out= -30,in=100, looseness=1] (-0.12,2)
(-0.12,2)to[out= -80,in=40, looseness=1] (0.15,1.7)
(0.15,1.7)to[out= -140,in=90, looseness=1.3](0,1.1)
(0,1.1)--(-.2,1.35)
(0,1.1)--(+.2,1.35);
\draw[color=black!50!white,rotate=90,shift={(0,-16.9)}, scale=0.6]
(-0.05,2.65)to[out= -90,in=150, looseness=1] (0.17,2.3)
(0.17,2.3)to[out= -30,in=100, looseness=1] (-0.12,2)
(-0.12,2)to[out= -80,in=40, looseness=1] (0.15,1.7)
(0.15,1.7)to[out= -140,in=90, looseness=1.3](0,1.1)
(0,1.1)--(-.2,1.35)
(0,1.1)--(+.2,1.35);
\draw[color=black, domain=-3.141: 3.141,
smooth,variable=\t,shift={(-1.72,0)},rotate=0, scale=0.9]plot({2.*sin(\t r)},
{2.*cos(\t r)}) ;
\draw[color=black,domain=-3.141: 3.141,
smooth,variable=\t,shift={(3.28,0)},rotate=0, scale=0.9]plot({2.*sin(\t r)},
{2.*cos(\t r)}) ;
\draw[color=black,domain=-3.141: 3.141,
smooth,variable=\t,shift={(8.28,0)},rotate=0, scale=0.9]plot({2.*sin(\t r)},
{2.*cos(\t r)}) ;
\draw[color=black,domain=-3.141: 3.141,
smooth,variable=\t,shift={(13.28,0)},rotate=0, scale=0.9]plot({2.*sin(\t r)},
{2.*cos(\t r)}) ;
\draw[color=black,domain=-3.141: 3.141,
smooth,variable=\t,shift={(18.28,0)},rotate=0, scale=0.9]plot({2.*sin(\t r)},
{2.*cos(\t r)}) ;
\draw[color=black,shift={(0,0.3)}] 
(-1.73,-1.8) 
to[out= 180,in=180, looseness=1] (-2.8,0) 
to[out= 60,in=150, looseness=1.5] (-1.5,1) 
(-2.8,0)
to[out=-60,in=180, looseness=0.9] (-1.25,-0.75)
(-1.5,1)
to[out= -30,in=90, looseness=0.9] (-1,0)
to[out= -90,in=60, looseness=0.9] (-1.25,-0.75)
to[out= -60,in=0, looseness=0.9](-1.73,-1.8);
\draw[color=black,shift={(5.3,0)}, scale=1.3] 
(-2.2,-0.6) 
to[out= 180,in=-110, looseness=1.5] (-2.5,0.5) 
to[out= 70,in=150, looseness=1.5] (-1.3,1) 
(-2.2,-0.6) 
to[out=60,in=180, looseness=0.9] (-2,0.5)
to[out=0,in=0, looseness=1](-2.2,-0.6) 
(-1.3,1)
to[out= -30,in=90, looseness=0.9] (-0.7,0)
to[out= -90,in=-120, looseness=0.9] 
(-2.2,-0.6) ;
\draw[color=black,shift={(15,0)}] 
(-2.2,-0.6) 
to[out= 180,in=-110, looseness=1.5] (-2.5,0.5) 
to[out= 70,in=150, looseness=1.5] (-1,1) 
(-2.2,-0.6) 
to[out=120,in=180, looseness=0.9] (-2,0.5)
to[out=0,in=0, looseness=1](-2.2,-0.6) 
(-1,1)
to[out= -30,in=90, looseness=0.9] (-0.7,0)
to[out= -90,in=-60, looseness=0.9] 
(-2.2,-0.6) ;
\draw[color=black,shift={(8.3,1.5)}, rotate=90, scale=0.85]
(-3,0)
to[out= 170,in=140, looseness=1] (-2.1,1.4)
to[out= -40,in=90, looseness=1] (0,0)
(-3,0)
to[out= -70,in=-180, looseness=1] (-1,-1.3)
to[out= 0,in=-90, looseness=1] (0,0)
(-3,0)
to[out= 50,in=180, looseness=1] (-2.3,0)
to[out= 60,in=150, looseness=1.5] (-1.75,1)
(-2.3,0)
to[out= -60,in=-120, looseness=0.9] (-1.5,-0.75)
(-1.75,1)
to[out= -30,in=90, looseness=0.9] (-1.25,0)
to[out= -90,in=60, looseness=0.9] (-1.5,-0.75);
\draw[color=black,shift={(20,-0.5)}, scale=0.9] 
(-1.73,1.7) 
to[out= 180,in=90, looseness=1](-2.8,0.85)
to[out=-90 ,in=-120, looseness=1](-2.2,0)
to[out= 120,in=150, looseness=1.5] (-1.5,1) 
(-2.2,0)
to[out=0,in=180, looseness=0.9] (-1.4,0)
(-1.5,1)
to[out= -30,in=90, looseness=0.9] (-1.3,0.5)
to[out= -90,in=60, looseness=0.9] (-1.4,0)
to[out= -60,in=-90, looseness=0.9](-0.9,0.85)
to[out= 90,in=0, looseness=0.9](-1.73,1.7);
\end{tikzpicture}
\end{center}
\begin{caption}{An example of evolution from a theta--shaped network to an Eyeglasses 
``type B'' passing by a $4$--point formation.}
\end{caption}
\end{figure}

\begin{rem}
As we said in Remark~\ref{nosym}, all these transitions of the networks between these ``dual'' topological shapes are not reversible in time. The angles between the curves are continuous as $t\to T^-$, discontinuous as $t\to T^+$.
\end{rem}

\bibliographystyle{amsplain}
\bibliography{2tripunti}

\providecommand{\bysame}{\leavevmode\hbox to3em{\hrulefill}\thinspace}
\providecommand{\MR}{\relax\ifhmode\unskip\space\fi MR }
\providecommand{\MRhref}[2]{%
  \href{http://www.ams.org/mathscinet-getitem?mr=#1}{#2}
}
\providecommand{\href}[2]{#2}
\begin{thebibliography}{10}

\bibitem{ablang1}
U.~Abresch and J.~Langer, \emph{The normalized curve shortening flow and
  homothetic solutions}, J. Diff. Geom. \textbf{23} (1986), no.~2, 175--196.

\bibitem{balhausman}
P.~Baldi, E.~Haus, and C.~Mantegazza, \emph{{Non--existence of Theta--shaped
  self--similarly} shrinking networks moving by curvature}, ArXiv Preprint
  Server -- http:/$\!\!$/$\!$arxiv.org, 2016.

\bibitem{brakke}
K.~A. Brakke, \emph{The motion of a surface by its mean curvature}, Princeton
  University Press, NJ, 1978.

\bibitem{bronsard}
L.~Bronsard and F.~Reitich, \emph{On three-phase boundary motion and the
  singular limit of a vector--valued {G}inzburg--{L}andau equation}, Arch. Rat.
  Mech. Anal. \textbf{124} (1993), no.~4, 355--379.

\bibitem{chenguo}
X.~Chen and J.-S. Guo, \emph{Self--similar solutions of a 2--{D}
  multiple--phase curvature flow}, Phys. D \textbf{229} (2007), no.~1, 22--34.

\bibitem{chzh}
K.-S. Chou and X.-P. Zhu, \emph{Shortening complete plane curves}, J. Diff.
  Geom. \textbf{50} (1998), no.~3, 471--504.

\bibitem{eckhui2}
K.~Ecker and G.~Huisken, \emph{Interior estimates for hypersurfaces moving by
  mean curvature}, Invent. Math. \textbf{105} (1991), no.~3, 547--569.

\bibitem{gaha1}
M.~Gage and R.~S. Hamilton, \emph{The heat equation shrinking convex plane
  curves}, J. Diff. Geom. \textbf{23} (1986), 69--95.

\bibitem{gray1}
M.~A. Grayson, \emph{The heat equation shrinks embedded plane curves to round
  points}, J. Diff. Geom. \textbf{26} (1987), 285--314.

\bibitem{hamilton2}
R.~S. Hamilton, \emph{Four--manifolds with positive curvature operator}, J.
  Diff. Geom. \textbf{24} (1986), no.~2, 153--179.

\bibitem{hamilton3}
\bysame, \emph{Isoperimetric estimates for the curve shrinking flow in the
  plane}, Modern methods in complex analysis (Princeton, NJ, 1992), Princeton
  University Press, NJ, 1995, pp.~201--222.

\bibitem{haettenschweiler}
J.~H{\"a}ttenschweiler, \emph{Mean curvature flow of networks with triple
  junctions in the plane}, Master's thesis, ETH Z\"urich, 2007.

\bibitem{huisk3}
G.~Huisken, \emph{Asymptotic behavior for singularities of the mean curvature
  flow}, J. Diff. Geom. \textbf{31} (1990), 285--299.

\bibitem{huisk2}
\bysame, \emph{A distance comparison principle for evolving curves}, Asian J.
  Math. \textbf{2} (1998), 127--133.

\bibitem{Ilnevsch}
T.~Ilmanen, A.~Neves, and F.~Schulze, \emph{On short time existence for the
  planar network flow}, ArXiv Preprint Server -- http:/$\!\!$/$\!$arxiv.org,
  2014.

\bibitem{MMN13}
A.~Magni, C.~Mantegazza, and M.~Novaga, \emph{Motion by curvature of planar
  networks {II}}, Ann. Sc. Norm. Sup. Pisa \textbf{15} (2016), 117--144.

\bibitem{Manlib}
C.~Mantegazza, \emph{Lecture notes on mean curvature flow}, Progress in
  Mathematics, vol. 290, Birkh\"auser/Springer Basel AG, Basel, 2011.

\bibitem{mannovplusch}
C.~Mantegazza, M.~Novaga, A.~Pluda, and F.~Schulze, \emph{Evolution of networks
  with multiple junctions}, in preparation.

\bibitem{mannovtor}
C.~Mantegazza, M.~Novaga, and V.~M. Tortorelli, \emph{Motion by curvature of
  planar networks}, Ann. Sc. Norm. Sup. Pisa \textbf{3 (5)} (2004), 235--324.

\bibitem{mazsae}
R.~Mazzeo and M.~S{\'a}ez, \emph{Self--similar expanding solutions for the
  planar network flow}, Analytic aspects of problems in {R}iemannian geometry:
  elliptic {PDE}s, solitons and computer imaging, S\'emin. Congr., vol.~22,
  Soc. Math. France, Paris, 2011, pp.~159--173.

\bibitem{pluda}
A.~Pluda, \emph{Evolution of spoon--shaped networks}, Network and Heterogeneus
  Media \textbf{11} (2016), no.~3, 509--526.

\bibitem{schnurerlens}
O.~C. Schn{\"u}rer, A.~Azouani, M.~Georgi, J.~Hell, J.~Nihar, A.~Koeller,
  T.~Marxen, S.~Ritthaler, M.~S{\'a}ez, F.~Schulze, and B.~Smith,
  \emph{Evolution of convex lens--shaped networks under the curve shortening
  flow}, Trans. Amer. Math. Soc. \textbf{363} (2011), no.~5, 2265--2294.

\bibitem{white1}
B.~White, \emph{A local regularity theorem for mean curvature flow}, Ann. of
  Math. (2) \textbf{161} (2005), no.~3, 1487--1519.

\end{thebibliography}
\end{document}